\documentclass[a4paper,11pt]{article}
\usepackage[utf8]{inputenc}
\usepackage[T1]{fontenc}
\usepackage{amsthm}
\usepackage{mathtools}
\usepackage{lmodern}
\usepackage{microtype}
\usepackage[english]{babel}

\usepackage[margin=2.5cm]{geometry}
\usepackage{enumitem}

\usepackage{tikz}
	\tikzstyle{lien}=[->,>=stealth,rounded corners=2pt,dashed]
	\tikzset{individu/.style={draw,#1},individu/.default={}}

\usepackage{xcolor}
	\colorlet{bleu}{blue!30!black}
	\colorlet{rouge}{red!30!black}
	\colorlet{vert}{green!20!black}
	\colorlet{jaune}{yellow!50!black}

\usepackage{marvosym}

\usepackage{hyperref}
\hypersetup{
	colorlinks=true,
		citecolor=bleu,
		linkcolor=rouge,
		urlcolor=vert,
		filecolor=jaune,
	breaklinks=true,
	bookmarksopen=true,
	bookmarksnumbered,
	pdftitle = {Fires on large recursive trees},
	pdfauthor = {Cyril Marzouk}
}

\newcommand{\ensembles}[1]{\mathbf{#1}}
	\newcommand{\N}{\ensembles{N}}
	
	\newcommand{\R}{\ensembles{R}}

\renewcommand{\P}{{\rm P}}
\newcommand{\E}{{\rm E}}

\newcommand{\ex}{\mathrm{e}}
\newcommand{\e}{\epsilon}
\renewcommand{\d}{\mathrm{d}}
\newcommand{\Card}{\mathrm{Card}}
\renewcommand{\t}{\mathsf{T}}
\newcommand{\tn}{\t_n}
\newcommand{\Cut}{\mathsf{Cut}}

\newcommand{\T}{\mathcal{T}}
\newcommand{\Red}{\mathcal{R}}
\newcommand{\Leb}{\mathrm{Leb}}

\newcommand{\cv}{\quad\mathop{\longrightarrow}^{}_{n \to \infty}\quad}
\newcommand{\cvk}{\quad\mathop{\longrightarrow}^{}_{k \to \infty}\quad}
\newcommand{\cvloi}{\quad\mathop{\longrightarrow}^{(d)}_{n \to \infty}\quad}

\theoremstyle{plain}
	\newtheorem{thm}{Theorem}
	\newtheorem{pro}{Proposition}
	\newtheorem{cor}{Corollary}
	\newtheorem{lem}{Lemma}
\theoremstyle{definition}
	\newtheorem{rem}{Remark}

\makeatletter \renewenvironment{proof}[1][\proofname] {\par\pushQED{\qed}\normalfont\topsep6\p@\@plus6\p@\relax\trivlist\item[\hskip\labelsep\bfseries#1\@addpunct{.}]\ignorespaces}{\popQED\endtrivlist\@endpefalse} \makeatother

\newenvironment{proofof}[1]{%
  \begin{proof}[Proof of #1]
}{%
  \end{proof}%
}

\title{Fires on large recursive trees}
\author{Cyril Marzouk
	\thanks{Institut f\"{u}r Mathematik, Universit\"{a}t Z\"{u}rich, Winterthurerstrasse 190, CH-8057 Zürich, Switzerland. Email: \href{mailto:cyril.marzouk@math.uzh.ch}{\texttt{cyril.marzouk@math.uzh.ch}}.}
	}
\date{}
\begin{document}

\maketitle

\begin{abstract}
We consider random dynamics on a uniform random recursive tree with $n$ vertices. Successively, in a uniform random order, each edge is either set on fire with some probability $p_n$ or fireproof with probability $1-p_n$. Fires propagate in the tree and are only stopped by fireproof edges. We first consider the proportion of burnt and fireproof vertices as $n\to\infty$, and prove a phase transition when $p_n$ is of order $\ln n/n$. We then study the connectivity of the fireproof forest, more precisely the existence of a giant component. We finally investigate the sizes of the burnt subtrees.
\end{abstract}

\section{Introduction}\label{section1}

Given a connected graph with $n$ vertices and a number $p_n \in [0,1]$, we consider the following random dynamics: initially every edge is flammable, then successively, in a uniform random order, each edge is either fireproof with probability $1 - p_n$ or set on fire with 
probability $p_n$; in the latter case, the edge burns, sets on fire its flammable neighbors and the fire propagates instantly in the graph, only stopped by fireproof edges. An edge which has burnt because of the propagation of fire is not subject to the dynamics thereafter. The dynamics continue until all edges are either burnt or fireproof. A vertex is called fireproof if all its adjacent edges are fireproof and called burnt otherwise. We discard the fireproof edges with at least one burnt extremity and thus get two families of subgraphs: one consists of fireproof subgraphs and the other of burnt subgraphs; see Figure \ref{fig:arbre_brule} for an illustration. We study the asymptotic behavior of the size of these two families of subgraphs and of their connected components as the size of the graph tends to infinity.

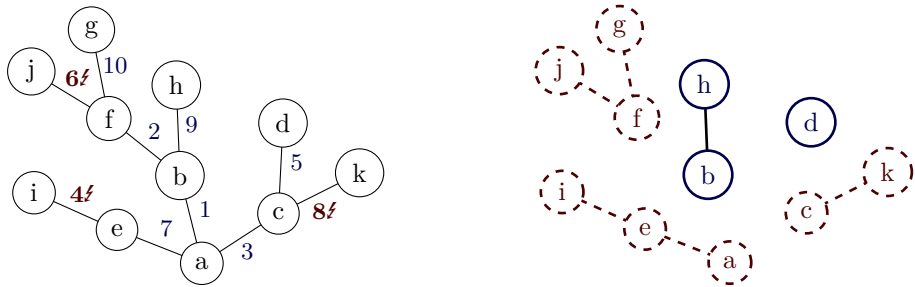
\begin{figure}[ht] \centering
\hfill
\begin{small}
\begin{tikzpicture}[scale=1.2]
\draw
	(-0.93,0.37) -- (-1.84,0.78)		
	(-1.02,1.6) -- (-1.87,2.12)		
	(0.84,0.55) -- (1.73,1)		
;
\draw
	(0,0) -- (-0.24,0.97)			
	(-0.24,0.97) -- (-1.02,1.6)		
	(0,0) -- (0.84,0.55)			
	(0.84,0.55) -- (0.89,1.55)		
	(0,0) -- (-0.93,0.37)			
	(-0.24,0.97) -- (-0.28,1.97)		
	(-1.02,1.6) -- (-1.21,2.58)		
;
\draw
	(0,0)			node [circle, fill=white, draw]	{a}
	(-0.24,0.97)	node [circle, fill=white, draw]	{b}
	(0.84,0.55)	node [circle, fill=white, draw]	{c}
	(0.89,1.55)	node [circle, fill=white, draw]	{d}
	(-0.93,0.37)	node [circle, fill=white, draw]	{e}
	(-1.02,1.6)		node [circle, fill=white, draw]	{f}
	(-1.21,2.58)	node [circle, fill=white, draw]	{g}
	(-0.28,1.97)	node [circle, fill=white, draw]	{h}
	(-1.84,0.78)	node [circle, fill=white, draw]	{i}
	(-1.87,2.12)	node [circle, fill=white, draw]	{j}
	(1.73,1)		node [circle, fill=white, draw]	{k}
;
\begin{footnotesize}
\color{bleu}{
\draw
	(0.04,0.59)	node		{1}
	(-0.52,1.46)	node		{2}
	(0.5,0.13)		node		{3}
	(1.04,1.11)	node		{5}
	(-0.39,0.38)	node		{7}
	(-0.11,1.54)	node		{9}
	(-0.95,2.19)	node		{10}
;}
{\bf \color{rouge}{
\draw
	(-1.31,0.77)	node		{4\Lightning}
	(-1.36,2.05)	node		{6\Lightning}
	(1.35,0.58)	node		{8\Lightning}
;}}
\end{footnotesize}
\end{tikzpicture}
\hfill
\begin{tikzpicture}[scale=1.2, line width=1pt]
\draw
	(-0.24,0.97) -- (-0.28,1.97)		
;
\draw [dashed, color=rouge]
	(-0.93,0.37) -- (-1.84,0.78)		
	(0,0) -- (-0.93,0.37)			
	(-1.02,1.6) -- (-1.87,2.12)		
	(-1.02,1.6) -- (-1.21,2.58)		
	(0.84,0.55) -- (1.73,1)		
;
\draw
	(0,0)			node [circle, color=rouge, fill=white, draw, dashed]	{a}
	(-0.24,0.97)	node [circle, line width=1pt, color=bleu, fill=white, draw]	{b}
	(0.84,0.55)	node [circle, color=rouge, fill=white, draw, dashed]	{c}
	(0.89,1.55)	node [circle, line width=1pt, color=bleu, fill=white, draw]	{d}
	(-0.93,0.37)	node [circle, color=rouge, fill=white, draw, dashed]	{e}
	(-1.02,1.6)		node [circle, color=rouge, fill=white, draw, dashed]	{f}
	(-1.21,2.58)	node [circle, color=rouge, fill=white, draw, dashed]	{g}
	(-0.28,1.97)	node [circle, line width=1pt, color=bleu, fill=white, draw]	{h}
	(-1.84,0.78)	node [circle, color=rouge, fill=white, draw, dashed]	{i}
	(-1.87,2.12)	node [circle, color=rouge, fill=white, draw, dashed]	{j}
	(1.73,1)		node [circle, color=rouge, fill=white, draw, dashed]	{k}
;
\end{tikzpicture}
\end{small}
\hfill{}
\caption{Given a recursive tree and an enumeration of its edges on the left, if the edges set on fire are the 4th, the 6th and the 8th, we get the two forests on the right where the burnt components are drawn with dotted lines and the fireproof ones with plain lines.}
\label{fig:arbre_brule}
\end{figure}

Fires on a graph find applications in statistical physics and in the study of epidemics propagating in a network. If the graph models a network, then the fires may be thought of as infections, and burnt and fireproof vertices respectively as infected and immune nodes. We stress that in the present model, infected nodes do not recover, we talk about a Susceptible-Infected-Removed epidemic, as opposed to Susceptible-Infected-Susceptible epidemics described by usual fire forest models in which the infected nodes recover and may be infected again later, see e.g. Drossel \& Schwabl \cite{Drossel_Schwabl-Self_organized_critical_forest-fire_model}. Another fire model in which burnt components are removed was studied by R\'{a}th \cite{Rath-Mean_field_frozen_percolation} and R\'{a}th \& T\'{o}th \cite{Rath_Toth-Erdos_Renyi_random_graphs_forest_fires_self_organized_criticality} who considered an Erd\H{o}s--Renyi graph in which edges appear at unit rate and the vertices are set on fire at small rate; when a vertex is set on fire, the whole connected component which contains it burns and is removed from the graph (edges and vertices). This model is in some sense dual to the present one: all edges are present at the beginning but fireproof edges act as barriers that stop the propagation of fires.

The model we investigate was introduced by Bertoin \cite{Bertoin-Fires_on_trees} and further studied in \cite{Marzouk-Fires_large_cayley_trees} in the special case where the graphs are Cayley trees of size $n$, i.e. picked uniformly at random amongst the $n^{n-2}$ different trees on a set of $n$ labelled vertices. In this paper, we first work in the general tree-setting: $(T_n)_{n \ge 1}$ is a sequence of random trees, where $T_n$ has size $n$. Our first result gives a criterion depending on the law of $T_n$ in order to observe a phase transition for the proportion of fireproof vertices, in the sense that for $p_n$ smaller than some specified sequence, the proportion of fireproof vertices in $T_n$ converges to $1$ in probability, for $p_n$ larger, it converges to $0$, and for $p_n$ comparable to this sequence, there is convergence in distribution to a non-trivial limit. Further, in the latter case, under a stronger assumption, we prove the joint convergence of the proportion of fireproof vertices and the sizes of the burnt connected components rescaled by $n$. See respectively Proposition \ref{pro:transition_de_phase_generale} and Proposition \ref{pro:limite_foret_ignifugee_et_arbres_brules}. Both results rely on the cut-tree associated with $T_n$ as defined by Bertoin \cite{Bertoin-Fires_on_trees} and a certain point process defined in \cite{Marzouk-Fires_large_cayley_trees}.

We then investigate more specifically the case of random recursive trees. A tree on the set of vertices $[n] \coloneqq \{1, \dots,  n\}$ is called recursive if, when rooted at $1$, the sequence of vertices along any branch from the root to a leaf is increasing. There are $(n-1)!$ such trees and we pick one of them uniformly at random, that we simply call random recursive tree, and denote by $\tn$. A random recursive tree on $[n]$ can be inductively constructed by the following algorithm: we start with the singleton $\{1\}$, then for every $i = 2, \dots, n$, the vertex $i$ is added to the current tree by an edge $\{u_i, i\}$, where $u_i$ is chosen uniformly at random in $\{1, \dots, i-1\}$ and independently of the previous edges.

We shall see that random recursive trees fulfill the assumptions of Proposition \ref{pro:transition_de_phase_generale} (and Proposition \ref{pro:limite_foret_ignifugee_et_arbres_brules}) previously alluded, which reads as follows: denote by $I_n$ the number of fireproof vertices in $\tn$, then
\begin{enumerate}
\item\label{intro-transition_de_phase_arbres_recursifs_sup} If $\lim_{n \to \infty} n p_n/\ln n = 0$, then $\lim_{n \to \infty} n^{-1} I_n = 1$ in probability (supercritical regime).

\item\label{intro-transition_de_phase_arbres_recursifs_sub} If $\lim_{n \to \infty} n p_n/\ln n = \infty$, then $\lim_{n \to \infty} n^{-1} I_n = 0$ in probability (subcritical regime).

\item\label{intro-transition_de_phase_arbres_recursifs_crit} If $\lim_{n \to \infty} n p_n/\ln n = c \in (0,\infty)$, then $n^{-1} I_n$ converges in distribution to $\e_c \wedge 1$, where $\e_c$ is exponentially distributed with rate $c$ (critical regime).
\end{enumerate}
Using more precise information about the cut-tree of random recursive trees, in particular a coupling with a certain random walk due to Iksanov and Möhle \cite{Iksanov_Moehle-A_probabilistic_proof_of_a_weak_limit_law_for_the_number_of_cuts_needed_to_isolate_the_root_of_a_random_recursive_tree}, we improve the convergence \ref{intro-transition_de_phase_arbres_recursifs_sub} as follows.

\begin{thm}\label{thm:limite_nombre_sites_ignifuges_cas_sous_critique}
In the subcritical regime $1 \gg p_n \gg \ln n / n$, we have the convergence in distribution
\begin{equation*}
\frac{p_n}{\ln(1/p_n)} I_n \cvloi \e_1,
\end{equation*}
where $\e_1$ is an exponential random variable with rate $1$.
\end{thm}

We further study the connectivity in the fireproof forest. We show the following asymptotic estimate for the probability that the root and a uniform random vertex belong to the same fireproof subtree, in both the critical and supercritical cases.

\begin{thm}\label{thm:probabilite_site_uniforme_et_racine_dans_le_meme_arbre_ignifuge}
Let $c \in [0, \infty)$ and $p_n$ such that $\lim_{n \to \infty} n p_n / \ln n = c$. Let also $X_n$ be a uniform random vertex in $[n]$ independent of $\tn$ and the fire dynamics. Then the probability that $X_n$ and $1$ belong to the same fireproof subtree converges towards $\ex^{-c}$ as $n \to \infty$.
\end{thm}

This enables us to deduce estimates on the size of the largest fireproof component. We shall see that with high probability as $n \to \infty$,
\begin{enumerate}
\item in the supercritical regime, there exists a giant fireproof component of size $\sim n$,

\item in the subcritical regime, the largest fireproof component has size of order $p_n^{-1} \ll n / \ln n$,

\item in the critical regime, the largest fireproof component has size of order $n / \ln n$ if the root burns and $\sim n$ if the root is fireproof.
\end{enumerate}

Finally, we study the sizes of the burnt subtrees, in order of appearance, in the critical regime $p_n \sim c \ln n / n$. The one which contains the root (if the root burns) has size of order $n$, whereas for the others, the logarithm of their size, rescaled by $\ln n$ converges in distribution to a limit strictly smaller than $1$. More precisely, let $\gamma_0 = 0$ and consider a sequence $(\gamma_j - \gamma_{j-1})_{j \ge 1}$ of i.i.d. exponential random variables with rate $c$, so each $\gamma_j$ has the gamma (also called Erlang) distribution with parameters $j$ and $c$. Consider also, conditional on $(\gamma_j)_{j \ge 1}$, a sequence $(Z_j)_{j \ge 1}$ of independent random variables, where $Z_j$ is distributed as an exponential random variable with rate $\gamma_j$ conditioned to be smaller than $1$. For every $i \in \N$, denote by $\theta_{n,i}$ the time at which the $i$-th fire occurs and by $b_{n,i}$ the size of the corresponding burnt subtree of $\tn$.

\begin{thm}\label{thm:limite_premiers_arbres_brules_cas_critique}
Consider the critical regime $p_n \sim c \ln n / n$. We have for every $i \in \N$,
\begin{equation*}
\frac{\ln n}{n} (\theta_{n,1}, \dots, \theta_{n,i}) \cvloi (\gamma_1, \dots, \gamma_i).
\end{equation*}
Furthermore, for every $j \in \N$, the probability that the root burns with the $j$-th fire converges to
\begin{equation*}
\E\bigg[\ex^{-\gamma_j} \prod_{i=1}^{j-1} (1 - \ex^{-\gamma_i})\bigg]
\end{equation*}
as $n \to \infty$. Finally, on this event, for every $k \ge j+1$, the vector
\begin{equation*}
\bigg(\frac{\ln b_{n,1}}{\ln n}, \dots, \frac{\ln b_{n, j-1}}{\ln n}, \frac{b_{n,j}}{n}, \frac{\ln b_{n, j+1}}{\ln n}, \dots, \frac{\ln b_{n, k}}{\ln n}\bigg)
\end{equation*}
converges in distribution towards
\begin{equation*}
(Z_1, \dots, Z_{j-1}, \ex^{-\gamma_j}, Z_{j+1}, \dots, Z_k).
\end{equation*}
\end{thm}

This work leaves open the question of the total number of burnt vertices in the supercritical regime: as in Theorem \ref{thm:limite_nombre_sites_ignifuges_cas_sous_critique} for the subcritical one, one would ask for a convergence in distribution of $n-I_n$, rescaled by some sequence which is negligible compared to $n$. This holds true for Cayley trees, see \cite{Marzouk-Fires_large_cayley_trees}.

Let us finally mention that the present fire dynamics on trees is closely related to the problem of isolating nodes, first introduced by Meir \& Moon \cite{Meir_Moon-Cutting_down_random_trees}. In this model, one is given a tree $T_n$ of size $n$ and $k$ vertices (chosen randomly or deterministically), say, $u_1, \dots, u_k$; then the edges of $T_n$ are successively removed in a uniform random order, and at each step, if a connected component newly created does not contain any of the $k$ selected vertices, it is immediately discarded. This random dynamics eventually end when the graph is reduced to the $k$ selected singletons, we say that the $k$ vertices have been isolated. The main interest in \cite{Meir_Moon-Cutting_down_random_trees} and subsequent papers concerns the behavior of the (random) number of steps $X(T_n; u_1, \dots, u_k)$ of this algorithm, as $n \to \infty$ and the number of selected vertices $k$ is  fixed. We shall see that $X(T_n; u_1, \dots, u_k)$ is related to the fire dynamics on $T_n$. Indeed, if one sees fireproof edges as being removed from the tree, then a vertex is fireproof if and only if it is isolated. In the context of random recursive trees, we will rely in particular on results of Meir \& Moon \cite{Meir_Moon-Cutting_down_recursive_trees} who estimated the first two moments of $X(\tn; 1)$, Iksanov \& M\"{o}hle \cite{Iksanov_Moehle-A_probabilistic_proof_of_a_weak_limit_law_for_the_number_of_cuts_needed_to_isolate_the_root_of_a_random_recursive_tree} who proved a limit theorem for the latter using probabilistic argument (it was first derived by Drmota \emph{et al.} \cite{Drmota_Iksanov_Moehle_Roesler-A_limiting_distribution_for_the_number_of_cuts_needed_to_isolate_the_root_of_a_random_recursive_tree} with a different proof), as well as Bertoin \cite{Bertoin-The_cut_tree_of_large_recursive_trees} who obtained a limit theorem for $X(\tn; U_1, \dots, U_k)$ when $U_1, \dots, U_k$ are $k$ independent uniform random vertices of $\tn$ (which was first obtained by Kuba \& Panholzer \cite{Kuba_Panholzer-Multiple_isolation_of_nodes_in_recursive_trees}).

The rest of this paper is organized as follows: in Section \ref{section21}, we first state and prove the general phase transition by relating the fire dynamics on $T_n$ to a point process on its cut-tree, following the route of \cite{Marzouk-Fires_large_cayley_trees}. We then focus on random recursive trees. After recalling some known results in Section \ref{section23}, in particular a coupling of Iksanov and M\"{o}hle \cite{Iksanov_Moehle-A_probabilistic_proof_of_a_weak_limit_law_for_the_number_of_cuts_needed_to_isolate_the_root_of_a_random_recursive_tree}, we prove Theorem \ref{thm:limite_nombre_sites_ignifuges_cas_sous_critique} in Section \ref{section3}. Section \ref{section4} is devoted to the proof of Theorem \ref{thm:probabilite_site_uniforme_et_racine_dans_le_meme_arbre_ignifuge} and the existence of a giant fireproof component. Finally, we prove Theorem \ref{thm:limite_premiers_arbres_brules_cas_critique} in Section \ref{section5}.

\section{Cut-trees and fires}\label{section2}

\subsection{General results}\label{section21}

In this section, we study the fire dynamics on a general sequence of random trees $T_n$ with $n$ labelled vertices, say, $[n] = \{1, \dots, n\}$. Let us recall the definition of the cut-tree of a finite tree introduced by Bertoin \cite{Bertoin-Fires_on_trees}, as well as the point process on the latter which is related to the fire dynamics on the original tree as described in \cite{Marzouk-Fires_large_cayley_trees}. We associate with $T_n$ a random rooted binary tree $\Cut(T_n)$ with $n$ leaves which records the genealogy induced by the fragmentation of $T_n$: each vertex of $\Cut(T_n)$ corresponds to a subset (or block) of $[n]$, the root of $\Cut(T_n)$ is the entire set $[n]$ and its leaves are the singletons $\{1\}, \dots, \{n\}$. We remove successively the edges of $T_n$ in a uniform random order; at each step, a subtree of $T_n$ with set of vertices, say, $V$, splits into two subtrees with sets of vertices, say, $V'$ and $V''$ respectively; in $\Cut(T_n)$, $V'$ and $V''$ are the two offsprings of $V$. Notice that, by construction, the set of leaves of the subtree of $\Cut(T_n)$ that stems from some block coincides with this block. See Figure \ref{fig:definition_cut_tree} below for an illustration.

\begin{figure}[ht] \centering
\begin{small}
\hfill
\begin{tikzpicture}[scale = 1.2]
\draw
	(0,0) -- (-0.24,0.97)			
	(-0.24,0.97) -- (-1.02,1.6)		
	(0,0) -- (0.84,0.55)			
	(-0.93,0.37) -- (-1.84,0.78)		
	(0.84,0.55) -- (0.89,1.55)		
	(-1.02,1.6) -- (-1.87,2.12)		
	(0,0) -- (-0.93,0.37)			
	(0.84,0.55) -- (1.73,1)		
	(-0.24,0.97) -- (-0.28,1.97)		
	(-1.02,1.6) -- (-1.21,2.58)		
;
\draw
	(0,0)			node [circle, fill=white, draw]	{a}
	(-0.24,0.97)	node [circle, fill=white, draw]	{b}
	(0.84,0.55)	node [circle, fill=white, draw]	{c}
	(0.89,1.55)	node [circle, fill=white, draw]	{d}
	(-0.93,0.37)	node [circle, fill=white, draw]	{e}
	(-1.02,1.6)		node [circle, fill=white, draw]	{f}
	(-1.21,2.58)	node [circle, fill=white, draw]	{g}
	(-0.28,1.97)	node [circle, fill=white, draw]	{h}
	(-1.84,0.78)	node [circle, fill=white, draw]	{i}
	(-1.87,2.12)	node [circle, fill=white, draw]	{j}
	(1.73,1)		node [circle, fill=white, draw]	{k}
;
\begin{footnotesize}
\draw
	(0.04,0.59)	node		{1}
	(-0.52,1.46)	node		{2}
	(0.5,0.13)		node		{3}
	(-1.31,0.77)	node		{4}
	(1.04,1.11)	node		{5}
	(-1.36,2.05)	node		{6}
	(-0.39,0.38)	node		{7}
	(1.35,0.58)	node		{8}
	(-0.11,1.54)	node		{9}
	(-0.95,2.19)	node		{10}
;
\end{footnotesize}
\end{tikzpicture}
\hfill
\begin{tikzpicture}
[level 1/.style={sibling distance=4.5 cm},
level 2/.style={sibling distance=2.2 cm},
level 3/.style={sibling distance=1.2 cm},
level 4/.style={sibling distance=1 cm}]
\node[individu]{abcdefghijk} [grow'=up,level distance=.8cm]
	child{node[individu]{acdeik}
		child{node[individu]{aei}
			child{node[individu]{ae}
				child{node[individu=circle]{a}}
				child{node[individu=circle]{e}}
			}
			child{node[individu=circle]{i}}
		}
		child{node[individu]{cdk}
			child{node[individu]{ck}
				child{node[individu=circle]{c}}
				child{node[individu=circle]{k}}
			}
			child{node[individu=circle]{d}}
		}
	}
	child{node[individu]{bfghj}
		child{node[individu]{bh}
			child{node[individu=circle]{b}}
			child{node[individu=circle]{h}}
		}
		child{node[individu]{fgj}
			child{node[individu]{fg}
				child{node[individu=circle]{f}}
				child{node[individu=circle]{g}}
			}
			child{node[individu=circle]{j}}
		}
	}
;
\end{tikzpicture}
\hfill
\end{small}
\caption{A tree with the order of cuts on the left and the corresponding cut-tree on the right. The order of the children in the cut-tree is irrelevant and has been chosen for aesthetic reasons only.}
\label{fig:definition_cut_tree}
\end{figure}
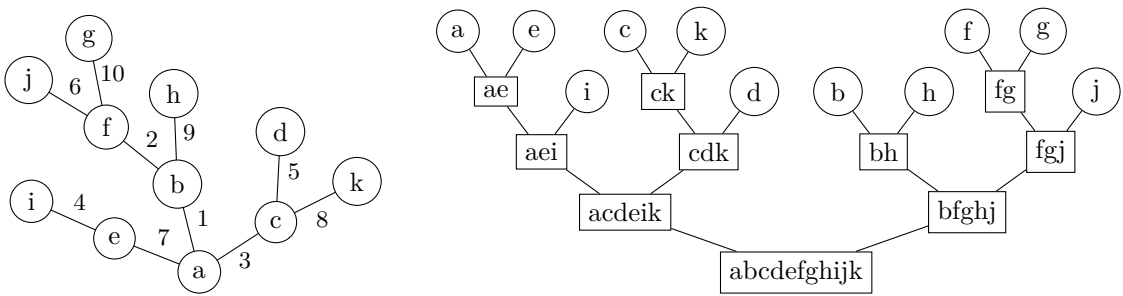

The cut-tree is an interesting tool to study the isolation of nodes in a tree described in the introduction. Indeed, fix $k$ vertices in the tree $T_n$, say, $u_1, \dots, u_k$, and consider the isolation procedure. It should be plain that the subtrees which are not discarded correspond to the blocks of the tree $\Cut(T_n)$ reduced to its branches from its root to the $k$ leaves $\{u_1\}, \dots, \{u_k\}$. As a consequence, the number of steps of this isolation procedure is given by the number of internal (i.e. non-singleton) nodes of this reduced tree or, equivalently, its length (i.e. number of edges) minus the number of distinct leaves plus one.

We endow $\Cut(T_n)$ with a mark process $\varphi_n$: at each generation, each internal block is marked independently of the others with probability $p_n$ provided that none of its ancestors has been marked, and not marked otherwise. This is equivalent to the following two-steps procedure: mark first every internal block independently with probability $p_n$, then along each branch from the root to a leaf, keep only the closest mark to the root and erase the other marks. Throughout this paper, ``marks on $\Cut(T_n)$'' shall always refer to the marks induced by $\varphi_n$.

Recall the fire dynamics on $T_n$ described in the introduction; if we remove each edge as soon as it is fireproof, then, when an edge is set on fire, it immediately burns the whole subtree which contains it. Observe that the only information which is lost with this point of view is the geometry of the fireproof forest. We can couple this dynamics on $T_n$ and the cut-tree $\Cut(T_n)$ endowed with the marks induced by $\varphi_n$ in such a way that the marked blocks of $\Cut(T_n)$ correspond to the burnt subtrees of $T_n$ and the leaves of $\Cut(T_n)$ which do not possess a marked ancestor correspond to the fireproof vertices of $T_n$, see Figure \ref{fig:foret_et_arbre_binaire_marque} for an illustration. We implicitly assume in the sequel that the fire dynamics on $T_n$ and the pair $(\Cut(T_n), \varphi_n)$ are coupled in this way.

\begin{figure}[ht] \centering
\hfill
\begin{small}
\begin{tikzpicture}[scale=1.2, line width=1pt]
\draw
	(-0.24,0.97) -- (-0.28,1.97)		
;
\draw [dashed, color=rouge]
	(-0.93,0.37) -- (-1.84,0.78)		
	(0,0) -- (-0.93,0.37)			
	(-1.02,1.6) -- (-1.87,2.12)		
	(-1.02,1.6) -- (-1.21,2.58)		
	(0.84,0.55) -- (1.73,1)		
;
\draw
	(0,0)			node [circle, color=rouge, fill=white, draw, dashed]	{a}
	(-0.24,0.97)	node [circle, line width=1pt, color=bleu, fill=white, draw]	{b}
	(0.84,0.55)	node [circle, color=rouge, fill=white, draw, dashed]	{c}
	(0.89,1.55)	node [circle, line width=1pt, color=bleu, fill=white, draw]	{d}
	(-0.93,0.37)	node [circle, color=rouge, fill=white, draw, dashed]	{e}
	(-1.02,1.6)		node [circle, color=rouge, fill=white, draw, dashed]	{f}
	(-1.21,2.58)	node [circle, color=rouge, fill=white, draw, dashed]	{g}
	(-0.28,1.97)	node [circle, line width=1pt, color=bleu, fill=white, draw]	{h}
	(-1.84,0.78)	node [circle, color=rouge, fill=white, draw, dashed]	{i}
	(-1.87,2.12)	node [circle, color=rouge, fill=white, draw, dashed]	{j}
	(1.73,1)		node [circle, color=rouge, fill=white, draw, dashed]	{k}
;
\end{tikzpicture}
\hfill
%
%
\begin{tikzpicture}
[level 1/.style={sibling distance=4.5 cm},
level 2/.style={sibling distance=2.2 cm},
level 3/.style={sibling distance=1.2 cm},
level 4/.style={sibling distance=1 cm},
base/.style={draw, line width=1pt, color=bleu},
feuille/.style={draw, circle},
feu/.style={draw, dashed, line width=1pt, color=rouge},
interne_brule/.style={draw, line width=1pt, color=rouge}
]
\node[base]{abcdefghijk} [grow'=up,level distance=.8cm]
	child[base]{node[base]{acdeik}
		child[feu]{node[interne_brule]{aei}
			child{node[interne_brule]{ae}
				child{node[feuille]{a}}
				child{node[feuille]{e}}
			}
			child{node[feuille]{i}}
		}
		child{node[base]{cdk}
			child[feu]{node[interne_brule]{ck}
				child{node[feuille]{c}}
				child{node[feuille]{k}}
			}
			child{node[feuille]{d}}
		}
	}
	child[base]{node[base]{bfghj}
		child{node[base]{bh}
			child{node[feuille]{b}}
			child{node[feuille]{h}}
		}
		child[feu]{node[interne_brule]{fgj}
			child{node[interne_brule]{fg}
				child{node[feuille]{f}}
				child{node[feuille]{g}}
			}
			child{node[feuille]{j}}
		}
	}
;
\end{tikzpicture}
\end{small}
\hfill{}
\caption{The forests after the dynamics on the recursive tree on the left and the corresponding cut-tree on the right: the plain blocks are the ones connected to the root, the dotted ones, those disconnected from the root. The leaves in the root-component of the cut-tree are the fireproof vertices, the other components code the burnt subtrees.}
\label{fig:foret_et_arbre_binaire_marque}
\end{figure}
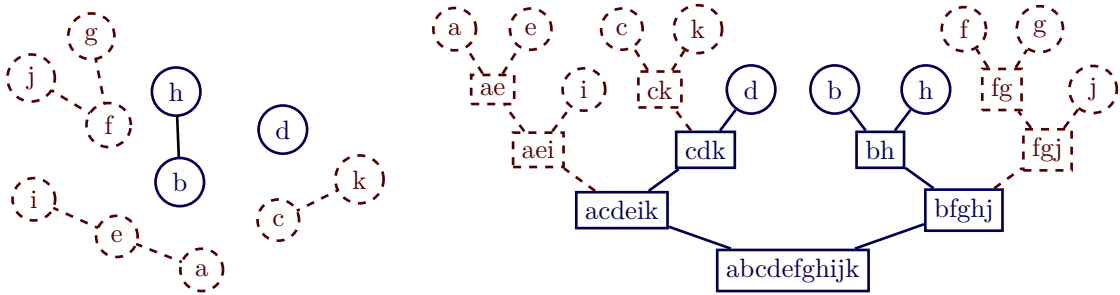

Fix $k \in \N$ and let $R_{n,k}$ be the tree $\Cut(T_n)$ reduced to its branches from its root to $k$ leaves chosen uniformly at random with replacement; denote by $L_{n, k}$ the length of $R_{n,k}$. Let $r : \N \to \R$ be some function such that $\lim_{n \to \infty} r(n) = \infty$. We introduce the following hypothesis:
\begin{equation}\label{eq:hypothese1}\tag{$H_k$}
r(n)^{-1} L_{n, k} \cvloi L_k,
\end{equation}
where $L_k$ is a non-negative and finite random variable. The arguments developed by Bertoin \cite{Bertoin-Fires_on_trees,Bertoin-Almost_giant_clusters_for_percolation_on_large_trees_with_logarithmic_heights} allow us to derive the following result.

\begin{pro}\label{pro:transition_de_phase_generale}
Denote by $I_n$ the number of fireproof vertices in $T_n$.
\begin{enumerate}
\item\label{pro:transition_de_phase_generale_sur} If $(H_1)$ holds and $\lim_{n \to \infty} r(n) p_n = 0$, then $n^{-1} I_n$ converges in probability to $1$.
\item\label{pro:transition_de_phase_generale_sous} If $(H_1)$ holds and $\lim_{n \to \infty} r(n) p_n = \infty$, then for every $\varepsilon > 0$, we have $\limsup_{n \to \infty} \P(I_n > \varepsilon n) \le \varepsilon^{-1}\P(L_1 = 0)$. In particular, if $L_1 > 0$ almost surely, then $n^{-1} I_n$ converges in probability to $0$.
\item\label{pro:transition_de_phase_generale_crit} If \eqref{eq:hypothese1} holds for every $k \in \N$ and $\lim_{n \to \infty} r(n) p_n = c$ with $c \in (0,\infty)$ fixed, then we have the convergence
\begin{equation*}
n^{-1} I_n \cvloi D(c),
\end{equation*}
where the law of $D(c)$ is characterized by its entire moments: for every $k \ge 1$,
\begin{equation}\label{eq:moments_loi_D_infini}
\E[D(c)^k] = \E[\exp(-cL_k)].
\end{equation}
Note that $\P(D(c) = 0) < 1$ and also $\P(D(c) = 1) = 1$ if and only if $\P(L_k = 0) = 1$.
\end{enumerate}
\end{pro}

\begin{proof}
Fix $k \in \N$; the variable $n^{-1} I_n$ represents the proportion of fireproof vertices in $T_n$, therefore its $k$-th moment is the probability that $k$ vertices of $T_n$ chosen uniformly at random with replacement and independently of the dynamics are fireproof. Using the coupling with $\Cut(T_n)$, the latter is the probability that there is no atom of the point process $\varphi_n$ on $R_{n,k}$. We thus have
\begin{equation*}
\E[(n^{-1} I_n)^k] = \E[(1-p_n)^{X_{n,k}}],
\end{equation*}
where $X_{n,k}$ denotes the number of internal nodes of $R_{n,k}$. Observe that $L_{n,k} - X_{n,k}$ is equal to the number of distinct leaves of $R_{n,k}$ minus one, which is bounded by $k-1$, then \eqref{eq:hypothese1} yields
\begin{equation*}
\lim_{n\to \infty} \E[(n^{-1} I_n)^k] = \E[\exp(-cL_k)]		\quad\text{when}\quad r(n) p_n \to c \in [0, \infty),
\end{equation*}
as well as
\begin{equation*}
\limsup_{n\to \infty} \E[(n^{-1} I_n)^k] \le \P(L_k=0)		\quad\text{when}\quad r(n) p_n \to \infty,
\end{equation*}
and the three assertions follow.
\end{proof}

The next proposition offers a reciprocal to Proposition \ref{pro:transition_de_phase_generale}, which shows that \eqref{eq:hypothese1}, $k \in \N$, form a necessary and sufficient condition for the critical case \ref{pro:transition_de_phase_generale_crit}. Observe that, if $D(c)$ is defined as in \eqref{eq:moments_loi_D_infini} for every $c \in (0, \infty)$, then $\lim_{c \to 0+} D(c) = 1$ in probability.

\begin{pro}\label{pro:reciproque_transition_de_phase_generale}
Assume that if $\lim_{n \to \infty} r(n) p_n = c \in (0, \infty)$, then $n^{-1} I_n$ converges in distribution as $n \to \infty$ to a limit $D(c)$ which satisfies and $\lim_{c \to 0+} D(c) = 1$ in probability. Then \eqref{eq:hypothese1} is fulfilled for every $k \in \N$ and the Laplace transform of $L_k$ is given by \eqref{eq:moments_loi_D_infini}.
\end{pro}

\begin{proof}
Fix $k \in \N$. The convergences $r(n) p_n \to c$ and $n^{-1} I_n \to D(c)$ in distribution imply that for every $c \in (0, \infty)$,
\begin{equation*}
\E[D(c)^k] = \lim_{n \to \infty} \E[(n^{-1} I_n)^k] = \lim_{n \to \infty} \E[(1-p_n)^{L_{n,k}}] = \lim_{n \to \infty} \E[\exp(-c r(n)^{-1} L_{n,k})].
\end{equation*}
Moreover, the assumption $\lim_{c \to 0+} D(c) = 1$ in probability implies that $\E[D(c)^k]$ converges to $1$ as $c \to 0+$. We conclude from Theorem XIII.1.2 of Feller \cite{Feller-An_introduction_to_probability_theory_and_its_applications_Volume_2} that the function $c \mapsto \E[D(c)^k]$ is the Laplace transform of a random variable $L_k$ and $r(n)^{-1} L_{n,k}$ converges in distribution to $L_k$.
\end{proof}

For the critical case $p_n \sim c / r(n)$, under a stronger assumption, we obtain the joint convergence in distribution of $I_n$ and the sizes of the burnt components, all rescaled by $n$. To this end, let us first introduce some notation. A compact rooted measure space is a quadruple $(\T, d, \rho, \mu)$ where $(\T, d)$ is a compact metric space, $\rho \in \T$ a distinguished element called the root of $\T$, and $\mu$ a Borel probability \emph{mass} measure on $\T$. This quadruple is called a (compact rooted) real tree if, furthermore, $\T$ is \emph{tree-like}, in the sense that it is a geodesic space for which any two points are connected via a unique continuous injective path up to re-parameterisation; see e.g. Le Gall \cite{LeGall-Random_trees_and_applications} and references therein for background. For each integer $k$, we denote by $\Red(\T, k)$ the tree $\T$ spanned by its root and $k$ i.i.d. elements chosen according to $\mu$.

Observe that $\Cut(T_n)$, rooted at $[n]$ and equipped with the graph distance and the uniform distribution on leaves, is a random real tree and $R_{n,k} = \Red(\Cut(T_n), k)$. Let $r : \N \to \R$ be some function such that $\lim_{n \to \infty} r(n) = \infty$. We introduce the following hypothesis:
\begin{equation}\label{eq:hypothese2}\tag{$H_k'$}
r(n)^{-1} R_{n,k} \cvloi \Red(\T, k).
\end{equation}
where $\T$ is a random rooted real tree and $r(n)^{-1} R_{n,k}$ has the same tree-structure as $R_{n,k}$ but with distances rescaled by a factor $r(n)^{-1}$. We adopt the framework of Aldous \cite{Aldous-The_continuum_random_tree_3}, viewing the reduced trees as a combinatorial rooted tree structure with edge lengths. Note that the fact that all the \eqref{eq:hypothese2}, $k \in \N$, hold is equivalent to the convergence of $r(n)^{-1} \Cut(T_n)$ to $\T$ for the so-called pointed Gromov--Prokhorov topology, see e.g. Gromov \cite{Gromov-Metric_structures_for_Riemannian_and_non_Riemannian_spaces} and Greven, Pfaffelhuber \& Winter \cite{Greven_Pfaffelhuber_Winter-Convergence_in_distribution_of_random_metric_measure_spaces} for references.

The distance on $\T$ induces an extra \emph{length} measure $\ell$, which is the unique $\sigma$-finite measure assigning measure $d(x,y)$ to the geodesic path between $x$ and $y$. We define on $\T$ the analog of the point process $\varphi_n$ on $\Cut(T_n)$: first sample a Poisson point process with intensity $c \ell(\cdot)$, then, along each branch from the root to a leaf, keep only the closest mark to the root (if any) and erase the other marks. The process $\Phi_c$ induces a partition of $\T$ in which two elements $x, y \in \T$ are connected if and only if there is no atom on the geodesic with extremities $x$ and $y$. Denote by $\# (\T, \Phi_c)$ the sequence given by the $\mu$-masses of each connected component of $\T$ after logging at the atoms of $\Phi_c$, the root-component first, and the next in non-increasing order.

Recall that $I_n$ denotes the total number of fireproof vertices in $T_n$ and let $b_{n,1}^\downarrow \ge b_{n,2}^\downarrow \ge \dots \ge 0$ be the sizes of the burnt subtrees, ranked in non-increasing order.

\begin{pro}\label{pro:limite_foret_ignifugee_et_arbres_brules}
If \eqref{eq:hypothese2} holds for every $k \in \N$, then in the regime $p_n \sim c / r(n)$, the convergence
\begin{equation*}
n^{-1} (I_n, b_{n,1}^\downarrow, b_{n,2}^\downarrow, \ldots) \cvloi \# (\mathcal{T}, \Phi_c)
\end{equation*}
holds in distribution for the $\ell^1$ topology.
\end{pro}

Remark that \eqref{eq:hypothese2} implies \eqref{eq:hypothese1} where $L_k$ is the total length of $\Red(\T, k)$; further the first element of $\# (\T, \Phi_c)$ is the variable $D(c)$ of Proposition \ref{pro:transition_de_phase_generale} and we can interpret identity \eqref{eq:moments_loi_D_infini} using $\T$. Sample $U_1, \dots, U_k \in \T$ independently according to $\mu$ and denote by $\Red(\T, k)$ the associated reduced tree. Denote also by $C_c$ the root-component of $\T$ after logging at the atoms of $\Phi_c$. Then  $D(c) = \mu(C_c)$ and
\begin{equation*}
\E[D(c)^k] = \E[\mu(C_c)^k] = \P(U_1,\dots, U_k \in C_c).
\end{equation*}
Since $U_1,\dots, U_k$ belong to $C_c$ if and only if there is no atom of the Poisson random measure on $\Red(\T, k)$, we also have
\begin{equation*}
\P(U_1,\dots, U_k \in C_c)
= \E[\exp(-c \ell(\Red(\T, k)))]
= \E[\exp(-c L_k)].
\end{equation*}

In the case where $T_n$ is a Cayley tree of size $n$, Bertoin \cite{Bertoin-Fires_on_trees} proved that the assumptions \eqref{eq:hypothese2}, $k \in \N$, hold with $r(n) = n^{1/2}$ and $\T$ the Brownian Continuum Random Tree. Moreover, a more explicit expression of the limit in Proposition \ref{pro:limite_foret_ignifugee_et_arbres_brules} in that case can be found in \cite{Marzouk-Fires_large_cayley_trees}. The proof of Proposition \ref{pro:limite_foret_ignifugee_et_arbres_brules} is essentially that of Lemma 2 in \cite{Marzouk-Fires_large_cayley_trees} to which we refer.

\subsection{The case of random recursive trees}\label{section22}

In the rest of this paper, we only consider the case where the trees are random recursive trees of size $n$ denoted by $\tn$. In this setting, Kuba \& Panholzer \cite[Theorem 3]{Kuba_Panholzer-Multiple_isolation_of_nodes_in_recursive_trees}, provide the assumptions \eqref{eq:hypothese1}, $k \in \N$:
\begin{equation}\label{eq:convergence_longueur_cut_tree_reduit_RRT}
\frac{\ln n}{n} L_{n,k} \cvloi \beta_k,
\end{equation}
where $\beta_k$ is a beta$(k, 1)$ random variable, i.e. with distribution $kx^{k-1} \d x$ on $[0, 1]$. Their result is stated in terms of the number of cuts needed to isolate uniform random vertices. Bertoin \cite{Bertoin-The_cut_tree_of_large_recursive_trees} proved an even stronger convergence than \eqref{eq:hypothese2} for every $k \in \N$. Indeed, the proof of Theorem 1 in \cite{Bertoin-The_cut_tree_of_large_recursive_trees} shows the following.

\begin{lem}[\cite{Bertoin-The_cut_tree_of_large_recursive_trees}]\label{lem:convergence_cut_tree_GHP}
Consider $\Cut(\tn)$ equipped with the metric $d_n$ given by the graph distance rescaled by a factor $\ln n / n$, and the uniform probability measure $\mu_n$ on the $n$ leaves. Then
\begin{equation*}
(\Cut(\tn), \{[n], \{1\}\}, d_n, \mu_n) \cv ([0, 1], \{0, 1\}, | \cdot |, \Leb)
\quad\text{in probability}
\end{equation*}
for the two-pointed Gromov--Hausdorff--Prokhorov topology, where $| \cdot |$ and $\Leb$ refer respectively to the Euclidean distance and Lebesgue measure.
\end{lem}

Propositions \ref{pro:transition_de_phase_generale} and \ref{pro:limite_foret_ignifugee_et_arbres_brules} then read as follows for recursive trees.

\begin{thm}\label{thm:transition_de_phase_arbres_recursifs}
We have for fires on $\tn$:
\begin{enumerate}[ref=\thethm(\roman{enumi})]
\item\label{thm:transition_de_phase_arbres_recursifs_sup} If $\lim_{n \to \infty} n p_n/\ln n = 0$, then $\lim_{n \to \infty} n^{-1} I_n = 1$ in probability.

\item\label{thm:transition_de_phase_arbres_recursifs_sub} If $\lim_{n \to \infty} n p_n/\ln n = \infty$, then $\lim_{n \to \infty} n^{-1} I_n = 0$ in probability.

\item\label{thm:transition_de_phase_arbres_recursifs_crit} If $\lim_{n \to \infty} n p_n/\ln n = c \in (0,\infty)$, then the convergence in distribution
\begin{equation*}
n^{-1} (I_n, b_{n,1}^\downarrow, b_{n,2}^\downarrow, \dots) \cvloi (\e_c \wedge 1, 1 - (\e_c \wedge 1), 0, 0, \dots)
\end{equation*}
holds for the $\ell^1$ topology, where $\e_c$ is an exponential random variable with rate $c$.
\end{enumerate}
\end{thm}

\begin{proof}
It follows from the convergences \eqref{eq:convergence_longueur_cut_tree_reduit_RRT} and Proposition \ref{pro:transition_de_phase_generale}, that $n^{-1} I_n$ converges in probability to $1$ when $p_n \ll \ln n / n$ and to $0$ when $p_n \gg \ln n / n$, and converges in distribution to $\e_c \wedge 1$ when $p_n \sim c \ln n / n$, since, for every $k \ge 1$, we have
\begin{equation*}
\E[\ex^{-c \beta_k}]
= \int_0^1 \ex^{-c x} k x^{k-1} \d x
= \ex^{-c} + \int_0^1 c \ex^{-c x} x^k \d x
= \E[(\e_c \wedge 1)^k].
\end{equation*}
Further, using Lemma \ref{lem:convergence_cut_tree_GHP}, the mark process $\Phi_c$ on $\T = [0,1]$ reduces to the point process with at most one mark, given by the smallest atom of a Poisson random measure on $[0, 1]$ with intensity $c \times \Leb$ if it exists, and no mark otherwise. The limit $\# (\T, \Phi_c)$ of Proposition \ref{pro:limite_foret_ignifugee_et_arbres_brules} is thus $(\e_c \wedge 1, 1 - (\e_c \wedge 1), 0, 0, \dots)$.
\end{proof}

Consider the critical regime of the fire dynamics on $\tn$, that is $p_n \sim c \ln n / n$ for some fixed $c \in (0, \infty)$. We identify the macroscopic burnt component of Theorem \ref{thm:transition_de_phase_arbres_recursifs_crit} above as the one containing the root. Denote by $b_n^0$ the size of the burnt subtree which contains the root of $\tn$ if the latter is burnt and $0$ otherwise, and $A_n$ the event that the root of $\tn$ burns.

\begin{pro}\label{pro:composante_brulee_macro_contient_la_racine}
In the critical regime $p_n \sim c \ln n / n$, we have
\begin{equation*}
(\mathbf{1}_{A_n}, n^{-1} I_n, n^{-1} b_n^0) \cvloi (\mathbf{1}_{\e_c < 1}, \e_c \wedge 1, 1 - (\e_c \wedge 1)),
\end{equation*}
where $\e_c$ is an exponential random variable with rate $c$.
\end{pro}

We see that the probability that the root burns converges to $1-\ex^{-c}$. Further, from this result and Theorem \ref{thm:transition_de_phase_arbres_recursifs_crit}, we see that in the critical regime, with high probability, when the root burns, its burnt component is the macroscopic one, and when it does not, there is no macroscopic burnt component. Finally, the density of fireproof vertices converges to $1$ in probability if we condition the root to be fireproof, and it converges in distribution to an exponential random variable with rate $c$ conditioned to be smaller than $1$ if we condition the root to burn.

\begin{proof}
On the path of $\Cut(\tn)$ from $[n]$ to $\{1\}$, there is at most one mark, at a height given by a geometric random variable with parameter $p_n \sim c \ln n / n$ if the latter is smaller than the height of $\{1\}$, and no mark otherwise. Furthermore, $b_n^0$ is equal to $0$ if there is no such mark and is given by the number of leaves of the subtree of $\Cut(\tn)$ that stems from this marked block otherwise. Thanks to Lemma \ref{lem:convergence_cut_tree_GHP}, this path and the mark converge in distribution to the interval $[0,1]$ with at most one mark, at distance $\e_c$ from $0$ if $\e_c < 1$, and no mark otherwise. Hence
\begin{equation*}
(\mathbf{1}_{A_n}, n^{-1} b_n^0) \cvloi (\mathbf{1}_{\e_c < 1}, 1 - (\e_c \wedge 1)).
\end{equation*}
Moreover, we already know from Theorem \ref{thm:transition_de_phase_arbres_recursifs_crit} that $n^{-1} I_n$ converges in distribution to $\e_c' \wedge 1$, where $\e_c'$ is exponentially distributed with rate $c$. Notice that we have $I_n \le n - b_n^0$, so $\e_c' \wedge 1 \le \e_c \wedge 1$ almost surely. Since they have the same law, we conclude that $\e_c' \wedge 1 = \e_c \wedge 1$ almost surely and the claim follows.
\end{proof}

In order to obtain more precise results on the fire dynamics on $\tn$, we need more information about $\Cut(\tn)$. We next recall some known results about the cut-tree of large random recursive trees, due to Meir and Moon \cite{Meir_Moon-Cutting_down_recursive_trees}, Iksanov and M{\"o}hle \cite{Iksanov_Moehle-A_probabilistic_proof_of_a_weak_limit_law_for_the_number_of_cuts_needed_to_isolate_the_root_of_a_random_recursive_tree},  and Bertoin \cite{Bertoin-The_cut_tree_of_large_recursive_trees}, and introduce the notation we shall use subsequently.

\subsection{The cut-tree of a random recursive tree}\label{section23}

Let $\zeta(n)$ be the length of the path in $\Cut(\tn)$ from its root $[n]$ to the leaf $\{1\}$. Set $C_{n, 0} \coloneqq [n]$ and for each $i = 1, \dots, \zeta(n)$, let $C_{n, i}$ and $C_{n, i}'$ be the two offsprings of $C_{n, i-1}$, with the convention that $1 \in C_{n, i}$; finally, denote by $\t_{n, i}$ and $\t_{n, i}'$ the subtrees of $\tn$ restricted to $C_{n, i}$ and $C_{n, i}'$ respectively. Note that for every $i \in \{1, \dots, \zeta(n)\}$, the collection $\{C_{n,1}', \dots, C_{n,i}', C_{n,i}\}$ forms a partition of $[n]$. The next lemma shows that the law of $\Cut(\tn)$ is essentially determined by that of the nested sequence $[n] = C_{n, 0} \supset C_{n,1} \supset \dots \supset C_{n, \zeta(n)} = \{1\}$.

Indeed, random recursive trees fulfill a certain consistency relation called splitting property or randomness preservation property. We extend the definition of a recursive tree to a tree on a totally ordered set of vertices which is rooted at the smallest element and such that the sequence of vertices along any branch from the root to a leaf is increasing. There is a canonical way to transform such a tree with size, say, $k$, to a recursive tree on $[k]$ by relabelling the vertices.

\begin{lem}\label{lem:propriete_fractale}
Fix $i \in \{1, \dots, \zeta(n)\}$. Conditional on the sets $C_{n, 1}', \dots, C_{n, i}'$ and $C_{n, i}$, the subtrees $\t_{n, 1}', \dots, \t_{n, i}'$ and $\t_{n, i}$ are independent random recursive trees on their respective set of vertices. Furthermore, conditional on these sets, the subtrees of $\Cut(\tn)$ that stem from these blocks are independent and distributed as cut-trees of random recursive trees on these sets.
\end{lem}

\begin{proof}
The first statement should be plain from the inductive construction of random recursive trees described in the introduction. The second follows since, in addition, if we restrict the fragmentation of $\tn$ described earlier to one of its subtree, the edges of this subtree are indeed removed in a uniform random order and this fragmentation is independent of the rest of $\tn$.
\end{proof}

As a consequence, we only need to focus on the size of the $C_{n,i}$'s. Our main tool relies on a coupling due originally to Iksanov \& M{\"o}hle \cite{Iksanov_Moehle-A_probabilistic_proof_of_a_weak_limit_law_for_the_number_of_cuts_needed_to_isolate_the_root_of_a_random_recursive_tree} that connects the latter with a certain random walk. Let us introduce a random variable $\xi$ with distribution
\begin{equation*}
\P(\xi = k) = \frac{1}{k(k+1)}, \qquad k \ge 1,
\end{equation*}
then a random walk
\begin{equation*}
S_j = \xi_1 + \dots + \xi_j, \qquad j \ge 1,
\end{equation*}
where $(\xi_i ; i \ge 1)$ are i.i.d. copies of $\xi$, and finally the last-passage time
\begin{equation*}
\lambda(n) = \max\{j \ge 1 : S_j < n\}.
\end{equation*}
A weaker form of the result in \cite{Iksanov_Moehle-A_probabilistic_proof_of_a_weak_limit_law_for_the_number_of_cuts_needed_to_isolate_the_root_of_a_random_recursive_tree}, which is sufficient for our purpose, is the following.

\begin{lem}\label{lem:couplage_Iksanov_Moehle}
One can construct on the same probability space a random recursive tree of size $n$ and its cut-tree, together with a version of the random walk $S$ such that $\zeta(n) \ge \lambda(n)$ and
\begin{equation}\label{eq:taille_sous_arbres_cut_tree}
(|C_{n,1}'|, \dots, |C_{n,\lambda(n)}'|, |C_{n, \lambda(n)}|) = (\xi_1, \dots, \xi_{\lambda(n)}, n - S_{\lambda(n)}).
\end{equation}
\end{lem}

From now on, we assume that the recursive tree $\tn$ and its cut-tree $\Cut(\tn)$ are indeed coupled with the random walk $S$. This coupling enables us to deduce properties of $\Cut(\tn)$ from that of $S$; we shall need the following results.

\begin{lem}\label{lem:proprietes_marche_aleatoire}
The random walk $S$ fulfills the following properties.
\begin{enumerate}[ref=\thelem(\roman{enumi})]
\item\label{lem:proprietes_marche_aleatoire_loi_des_grands_nombres} Weak law of large numbers:
\begin{equation*}
\frac{1}{k \ln k} S_k \cvk 1 \quad\text{in probability.}
\end{equation*}

\item\label{lem:proprietes_marche_aleatoire_limite_lambda}
The last-passage time satisfies
\begin{equation*}
\frac{\ln n}{n} \lambda(n) \cv 1 \quad\text{in probability.}
\end{equation*}

\item\label{lem:proprietes_marche_aleatoire_reste_apres_lambda}
The undershoot satisfies
\begin{equation*}
\frac{\ln n}{n} (n - S_{\lambda(n)}) \cv 0 \quad\text{in probability.}
\end{equation*}

\item\label{lem:proprietes_marche_aleatoire_limite_mesure}The random point measure
\begin{equation*}
\sum_{i=1}^{\lambda(n)} \delta_{\frac{\ln n}{n} \xi_i}(\d x)
\end{equation*}
converges in distribution on the space of locally finite measures on $(0,\infty]$ endowed with the vague topology towards to a Poisson random measure with intensity $x^{-2} \d x$.
\end{enumerate}
\end{lem}

\begin{proof}
The first assertion can be checked using generating functions; a standard limit theorem for random walk with step distribution in the domain of attraction of a stable law entails moreover the weak convergence of $k^{-1} S_k - \ln k$ to the so-called continuous Luria-Delbr\"{u}ck distribution, see for instance Geluk \& de Haan \cite{Geluk_de_Haan-Stable_probability_distributions_and_their_domains_of_attraction_a_direct_approach}. Iksanov \& M{\"o}hle \cite{Iksanov_Moehle-A_probabilistic_proof_of_a_weak_limit_law_for_the_number_of_cuts_needed_to_isolate_the_root_of_a_random_recursive_tree}, provide finer limit theorems for the last-passage time as well as the undershoot, see respectively Proposition 2 and Lemma 2 there. Finally, the last assertion is the claim of Lemma 1(ii) of Bertoin \cite{Bertoin-The_cut_tree_of_large_recursive_trees} and follows readily from the distribution of $\xi$ and Theorem 16.16 of Kallenberg \cite{Kallenberg-Foundations_of_modern_probability}.
\end{proof}

Note from Lemma \ref{lem:couplage_Iksanov_Moehle} that
\begin{equation*}
\lambda(n) \le \zeta(n) \le \lambda(n) + |C_{n, \lambda(n)}| = \lambda(n) + n - S_{\lambda(n)}.
\end{equation*}
Lemma \ref{lem:proprietes_marche_aleatoire_limite_lambda} and \ref{lem:proprietes_marche_aleatoire_reste_apres_lambda} thus entail
\begin{equation*}
\frac{\ln n}{n} \zeta(n) \cv 1\qquad\text{in probability.}
\end{equation*}
This result was obtained earlier by Meir \& Moon \cite{Meir_Moon-Cutting_down_recursive_trees} who proved
\begin{equation*}
\lim_{n \to \infty} \E\bigg[\frac{\ln n}{n} \zeta(n)\bigg]
= \lim_{n \to \infty} \E\bigg[\bigg(\frac{\ln n}{n} \zeta(n)\bigg)^2\bigg]
= 1.
\end{equation*}
We will use this stronger result in Section \ref{section4} below.

\section{Density of fireproof vertices}\label{section3}

As claimed in the introduction, we prove a non-trivial limit in distribution for the number $I_n$ of fireproof vertices in $\tn$, under an appropriate scaling, in the subcritical regime. We begin with a lemma.

\begin{lem}\label{lem:racine_brule_avec_grande_proba_cas_sous_critique}
Consider the subcritical regime $1 \gg p_n \gg \ln n / n$. Then, as $n \to \infty$, the root of $\tn$ burns with high probability, and the size of its burnt component, rescaled by $n$, converges to $1$ in probability.
\end{lem}

\begin{proof}
Consider the path from $[n]$ to $\{1\}$ in $\Cut(\tn)$. It contains at most one mark, whose height $\sigma(n)$ is distributed as $g_n \wedge \zeta(n)$ where $g_n$ is a geometric random variable with parameter $p_n$ independent of $\zeta(n)$. Recall from Lemma \ref{lem:proprietes_marche_aleatoire} that $\zeta(n) \ge \lambda(n) \sim n / \ln n$ in probability, so this mark exists with high probability and, moreover,
\begin{equation*}
p_n \sigma(n) \cvloi \e_1.
\end{equation*}
In particular, we have $\sigma(n) \le \lambda(n)$ with high probability. On this event, observe thanks to Lemma \ref{lem:couplage_Iksanov_Moehle} that the size of the burnt component which contains the root is given by
\begin{equation*}
|C_{n,\sigma(n)}| = n - S_{\sigma(n)}.
\end{equation*}
It follows from Lemma \ref{lem:proprietes_marche_aleatoire_loi_des_grands_nombres} and a standard argument (cf. Theorem 15.17 in Kallenberg \cite{Kallenberg-Foundations_of_modern_probability}) that for every $y \ge 0$,
\begin{equation*}
\sup_{x \in [0, y]} \bigg|\frac{p_n}{\ln(1/p_n)} S_{\lfloor x / p_n\rfloor} - x\bigg| \cv 0
\quad\text{in probability,}
\end{equation*}
and we conclude that
\begin{equation}\label{eq:convergence_arbres_non_brules_cas_sous_critique}
\frac{p_n}{\ln(1/p_n)} S_{\sigma(n)} \cvloi \e_1.
\end{equation}
Note that $p_n  / \ln(1/p_n) \gg 1/n$ when $p_n \gg \ln n / n$, therefore
\begin{equation*}
n^{-1} |C_{n,\sigma(n)}| \cv 1
\qquad\text{in probability,}
\end{equation*}
and the proof is complete.
\end{proof}

We next prove Theorem \ref{thm:limite_nombre_sites_ignifuges_cas_sous_critique}, specifically,
\begin{equation*}
\frac{p_n}{\ln(1/p_n)} I_n \cvloi \e_1 \qquad\text{when}\quad p_n \gg \ln n / n,
\end{equation*}
where $\e_1$ is an exponential random variable with rate $1$. As we noted, $p_n  / \ln(1/p_n) \gg 1/n$ when $p_n \gg \ln n / n$, so this result recovers Theorem \ref{thm:transition_de_phase_arbres_recursifs_sub}. Observe also that in the critical regime $p_n' \sim c \ln n / n$, we have $p_n' / \ln(1/p_n') \sim c/n$; it then follows from Theorem \ref{thm:transition_de_phase_arbres_recursifs_crit} that in this case
\begin{equation*}
\frac{p_n'}{\ln(1/p_n')} I_n \cvloi c (\e_c \wedge 1) = \e_1 \wedge c,
\end{equation*}
and the right-hand side further converges to $\e_1$ as $c \to \infty$. The same phenomenon appears in \cite{Marzouk-Fires_large_cayley_trees} for Cayley trees, were in both regimes, critical and subcritical, one should rescale $I_n$ by a factor $p_n^2$ and the limit for the critical case converges to that of the subcritical case when $c \to \infty$.

\begin{proof}[Proof of Theorem \ref{thm:limite_nombre_sites_ignifuges_cas_sous_critique}]
The proof borrows ideas from Section 3 of Bertoin \cite{Bertoin-The_cut_tree_of_large_recursive_trees}. Recall that $I_n$ is the number of leaves in the component of $\Cut(\tn)$ which contains the root $[n]$ after logging at the atom of $\varphi_n$. According to the proof of Lemma \ref{lem:racine_brule_avec_grande_proba_cas_sous_critique}, with high probability, there exists a mark on the path from $[n]$ to $\{1\}$ in $\Cut(\tn)$, at height $\sigma(n)$. Observe that all the other marks of $\varphi_n$ are contained in the subtrees of $\Cut(\tn)$ that stem from the blocks $C_{n,1}', \dots, C_{n,\sigma(n)}'$. Moreover, we have seen, appealing to Lemma \ref{lem:couplage_Iksanov_Moehle}, that
\begin{equation*}
\frac{p_n}{\ln(1/p_n)} \sum_{i = 1}^{\sigma(n)} |C_{n,i}'| \cvloi \e_1.
\end{equation*}
It only remains to show that the proportion of leaves in all these subtrees which belong to the root-component of $\Cut(\tn)$ converges to $1$ in probability. Recall from Lemma \ref{lem:propriete_fractale} that, conditionally given the sets $C_{n,1}', \dots, C_{n,\sigma(n)}'$, the subtrees of $\Cut(\tn)$ that stems from these blocks are independent and distributed respectively as the cut-tree of a random recursive tree on $C_{n,i}'$. As in the proof of Proposition \ref{pro:transition_de_phase_generale}, we show that the probability that a leaf chosen uniformly at random in these subtrees belongs to the root-component converges to $1$. The latter is bounded from below by
\begin{equation*}
\E\Big[(1-p_n)^{\max\{\mathrm{Depth}(\Cut(\t_{n,i}')) , 1 \le i \le \sigma(n)\}}\Big],
\end{equation*}
where $\mathrm{Depth}(T)$ denotes the maximal distance in the tree $T$ from the root to a leaf. The proof then boils down to the convergence
\begin{equation*}
p_n \max\{\mathrm{Depth}(\Cut(\t_{n,i}')) ; 1 \le i \le \sigma(n)\} \cv 0
\qquad \text{in probability.}
\end{equation*}
Proposition 1 of Bertoin \cite{Bertoin-The_cut_tree_of_large_recursive_trees} proves a similar statement, in the case where $p_n = \ln n / n$ and the maximum is up to $\lambda(n)$. We closely follow the arguments in \cite{Bertoin-The_cut_tree_of_large_recursive_trees}. Fix $\varepsilon > 0$; from Lemma \ref{lem:propriete_fractale}, since $\mathrm{Depth}(\Cut(T)) \le |T|$, for every $m \in \N$ and $a > 0$,
\begin{equation*}
\P(p_n \max\{\mathrm{Depth}(\Cut(\t_{n,i}')) , 1 \le i \le \sigma(n)\} > \varepsilon)
\end{equation*}
is bounded from above by
\begin{equation*}
m \sup_{k \le a / p_n} \P(p_n \mathrm{Depth}(\Cut(\t_{k})) > \varepsilon) + \P(N(\varepsilon, n) > m) + \P(N(a, n) \ge 1),
\end{equation*}
where $N(z, n) = \Card\{i = 1, \dots, \sigma(n) : |\t_{n, i}'| > z / p_n\}$. On the one hand, from \eqref{eq:taille_sous_arbres_cut_tree} and the distribution of $\xi$, conditionally given $\sigma(n)$ with $\sigma(n) \le \lambda(n)$, $N(z, n)$ is binomial distributed with parameters $\sigma(n)$ and $\lceil z / p_n \rceil^{-1}$; as a consequence, for any $\delta > 0$, we may fix $m$ and $a$ sufficiently large so that
\begin{equation*}
\limsup_{n \to \infty} \P(N(\varepsilon, n) > m) + \P(N(a, n) \ge 1) \le \delta.
\end{equation*}
On the other hand, from \eqref{eq:taille_sous_arbres_cut_tree}, we have
\begin{equation*}
\mathrm{Depth}(\Cut(\t_{k})) \le \lambda(k) + \max\{\xi_i , 1 \le i \le \lambda(k)\} + (k - S_{\lambda(k)})
\end{equation*}
which, rescaled by a factor $p_n$, converges in probability to $0$ uniformly for $k \le a / p_n$ thanks to Lemma \ref{lem:proprietes_marche_aleatoire}. This concludes the proof.
\end{proof}

\begin{rem}
Let $C_n$ be the root-component of $\Cut(\tn)$ after performing a Bernoulli bond percolation, in which each edge is removed with probability $p_n$; we endow it with the graph distance $d_n$ and the measure $\nu_n$ which assigns mass $1$ to each leaf. Adapting Section 3 of Bertoin \cite{Bertoin-The_cut_tree_of_large_recursive_trees}, the proofs of Proposition \ref{pro:composante_brulee_macro_contient_la_racine} and Theorem \ref{thm:limite_nombre_sites_ignifuges_cas_sous_critique} here respectively entail the following weak convergences for the pointed Gromov--Hausdorff--Prokhorov topology:
\begin{equation*}
\bigg(C_n, \{[n]\}, p_n d_n, \frac{p_n}{\ln(1/p_n)} \nu_n\bigg) \cvloi \big([0, \e_1], \{0\}, | \cdot |, \Leb\big),
\quad\text{when}\quad p_n \gg \ln n / n,
\end{equation*}
and
\begin{equation*}
\bigg(C_n, \{[n]\}, \frac{\ln n}{n} d_n, \frac{1}{n} \nu_n\bigg) \cvloi \big([0, \e_c \wedge 1], \{0\}, | \cdot |, \Leb\big),
\quad\text{when}\quad p_n \sim c \ln n / n,
\end{equation*}
where in both cases, $| \cdot |$ and $\Leb$ refer respectively to the Euclidean distance and Lebesgue measure. The same arguments also yield
\begin{equation*}
\bigg(C_n, \{[n]\}, \frac{\ln n}{n} d_n, \frac{1}{n} \nu_n\bigg) \cvloi \big([0, 1], \{0\}, | \cdot |, \Leb\big),
\quad\text{when}\quad p_n \ll \ln n / n.
\end{equation*}
\end{rem}

\section{Connectivity properties of the fireproof forest}\label{section4}

We next focus on the fireproof forest. We first find an asymptotic estimate for the probability that the root of $\tn$ and a uniform random vertex belong to the same fireproof subtree, in both the critical and supercritical cases. We then derive some consequences in terms of the existence of a giant fireproof component.

\begin{proofof}{Theorem \ref{thm:probabilite_site_uniforme_et_racine_dans_le_meme_arbre_ignifuge}}
We use a so-called \emph{spinal decomposition}: fix $X \in [n]$ and denote by $h(X) = d(X, 1)$ the height of $X$ in $\tn$. Let $V_0, \dots, V_{h(X)}$ be the vertices on the oriented branch from $1$ to $X$: $V_0 = 1$, $V_{h(X)} = X$ and for each $i = 1, \dots, h(X)$, $V_{i-1}$ is the parent of $V_i$. Removing all the edges $\{V_i, V_{i+1}\}$ disconnects $\tn$ into $h(X)+1$ subtrees denoted by $T_0, \dots, T_{h(X)}$ where $T_i$ contains $V_i$ for every $i = 0, \dots, h(X)$. Clearly, $V_0 = 1$ and $V_{h(X)} = X$ belong to the same fireproof connected component if and only if all the $V_i$'s are fireproof, i.e. when all the edges $\{V_i, V_{i+1}\}$ are fireproof and each $V_i$ is fireproof for the dynamics restricted to the tree $T_i$.

Using the inductive construction of random recursive trees described in the introduction, one sees that, when removing the edge $\{V_0, V_1\}$, the two subtrees we obtain are, conditional on their set of vertices, independent random recursive trees. The one containing $V_0$ is $T_0$. Removing the edge $\{V_1, V_2\}$ in the other subtree, we obtain similarly that $T_1$ is, conditional on its set of vertices and that of $T_0$, a random recursive tree independent of $T_0$. We conclude by induction that conditional on their set of vertices, the $T_i$'s are independent random recursive trees rooted at $V_i$ respectively. 

Recall that for every $k \ge 1$, $\zeta(k)$ denotes the height of the leaf $\{1\}$ in the cut-tree $\Cut(\t_k)$ of a random recursive tree of size $k$. We have seen that the root of $\t_k$ is fireproof with probability $\E[(1-p_n)^{\zeta(k)}]$. Thus, from the discussion above, the probability that $X$ and $1$ belong to the same fireproof connected component is given by
\begin{equation}\label{eq:proba_qu_une_branche_soit_ignifugee}
\E\bigg[\exp\bigg(\ln(1-p_n) \bigg(h(X) + \sum_{i=0}^{h(X)} \zeta_i(|T_i|)\bigg)\bigg)\bigg],
\end{equation}
where $(\zeta_i(k) ; k \ge 1)_{i \ge 0}$ is a sequence of i.i.d. copies of $(\zeta(k) ; k \ge 1)$. We prove that if $X_n$ is uniformly distributed on $[n]$, then
\begin{equation}\label{eq:nombre_de_coupes_pour_isoler_un_chemin}
\frac{\ln n}{n} \bigg(h(X_n) + \sum_{i=0}^{h(X_n)} \zeta_i(|T_i|)\bigg) \cv 1
\qquad\text{in probability,}
\end{equation}
which yields Theorem \ref{thm:probabilite_site_uniforme_et_racine_dans_le_meme_arbre_ignifuge}. It is well-known that $h(X_n) \sim \ln n$ in probability as $n \to \infty$ so we only need to consider the sum in \eqref{eq:nombre_de_coupes_pour_isoler_un_chemin}. Let us first discuss the distribution of the $|T_i|$'s.

Let $S_n(0)$ be a random variable uniformly distributed in $[n]$. Then, for every $i \ge 1$, conditionally given $\overline{S}_n(i-1) \coloneqq S_n(0) + \dots + S_n(i-1)$, let $S_n(i)$ be uniformly distributed in $[n-\overline{S}_n(i-1)]$ if $\overline{S}_n(i-1) < n$ and set $S_n(i) = 0$ otherwise. Let $\kappa_n \coloneqq \inf\{i \ge 0 : \overline{S}_n(i) = n\}$; note that $S_n(i) = 0$ if and only if $i > \kappa_n$ and that $S_n(0) + \dots + S_n(\kappa_n) = n$. We call the sequence $\mathbf{S}_n \coloneqq (S_n(0), \dots, S_n(\kappa_n))$ a \emph{discrete stick-breaking process}. Denote finally by $\tilde{S}_n$ a size-biased pick from $\mathbf{S}_n$. Then $\tilde{S}_n$ is uniformly distributed in $[n]$ (see Lemma \ref{lem:nombre_de_cycles_a_de_taille_k} below) and for every measurable and non-negative functions $f$ and $g$, we have
\begin{equation}\label{eq:stick_breaking_discret}
\E\bigg[g\bigg(\sum_{i=0}^{\kappa_n} f(S_n(i))\bigg)\bigg]
= \E\bigg[g\bigg(\sum_{i=0}^{\kappa_n} n \frac{f(S_n(i))}{S_n(i)} \P(\tilde{S}_n = S_n(i) \mid \mathbf{S}_n)\bigg)\bigg]
= \E\bigg[g\bigg(n\frac{f(\tilde{S}_n)}{\tilde{S}_n}\bigg)\bigg]
\end{equation}

The stick breaking-process appears in a random recursive tree in two ways: \emph{vertically} and \emph{horizontally}. Indeed, if we discard the root of $\t_{n+1}$ and its adjacent edges, then the sequence formed by the sizes of the resulting subtrees, ranked in increasing order of their root is distributed as $\mathbf{S}_n$. In particular, the one containing the leaf $n+1$ has size $\tilde{S}_n$ so is uniformly distributed in $[n]$. Further, if $n+1$ is not the root of this subtree, we can iterate the procedure of removing the root and discarding all the components but the one containing $n+1$. Conditionally given the size $s_i$ of the component containing $n+1$ at the $i$-th step, its size at the $i+1$-st step is uniformly distributed in $[s_i-1]$, thus defining a stick-breaking process. We continue until the component containing $n+1$ is reduced to the singleton $\{n+1\}$; this takes $h(n+1) = \kappa_n+1$ steps.

Let $X_n$ be the parent of $n+1$ in $\t_{n+1}$; then $X_n$ is distributed as a uniform random vertex of $\tn$. Moreover, we just saw that $h(X_n) = h(n+1)-1$ is distributed as $\kappa_n$ and, further, the sequence $|T_0|, \dots, |T_{h(X_n)}|$ previously defined is distributed as $\mathbf{S}_n$. Theorem \ref{thm:probabilite_site_uniforme_et_racine_dans_le_meme_arbre_ignifuge} will thereby follow from the convergence
\begin{equation}\label{eq:nombre_de_coupes_pour_isoler_des_racines_dans_un_stick_breaking}
\frac{\ln n}{n} \sum_{i=0}^{\kappa_n} \zeta_i(S_n(i)) \cv 1
\qquad\text{in probability.}
\end{equation}
We prove the convergence of the first and second moments. Let us define $f_1(\ell) \coloneqq \E[\zeta(\ell)]$ and $f_2(\ell) \coloneqq \E[\zeta(\ell)^2]$ for every $\ell \ge 1$. We already mentioned that Meir \& Moon \cite{Meir_Moon-Cutting_down_recursive_trees} proved that, as $\ell \to \infty$,
\begin{equation}\label{eq:moments_zeta}
f_1(\ell) = \frac{\ell}{\ln \ell} (1+o(1))
\qquad\text{and}\qquad
f_2(\ell) = \bigg(\frac{\ell}{\ln \ell}\bigg)^2 (1+o(1)).
\end{equation}
Conditioning first on $\mathbf{S}_n$ and then averaging, we have
\begin{equation*}
\E\bigg[\sum_{i=0}^{\kappa_n} \zeta_i(S_n(i))\bigg]
= \E\bigg[\sum_{i=0}^{\kappa_n} f_1(S_n(i))\bigg],
\end{equation*}
and, using the conditional independence of the $\zeta_i$'s,
\begin{align*}
\E\bigg[\bigg(\sum_{i=0}^{\kappa_n} \zeta_i(S_n(i))\bigg)^2\bigg]
&= \E\bigg[\sum_{i=0}^{\kappa_n} \zeta_i(S_n(i))^2\bigg] + \E\bigg[\sum_{i\ne j} \zeta_i(S_n(i)) \zeta_j(S_n(j))\bigg]
\\
&= \E\bigg[\sum_{i=0}^{\kappa_n} f_2(S_n(i))\bigg] + \E\bigg[\sum_{i\ne j} f_1(S_n(i)) f_1(S_n(j))\bigg]
\\
&= \E\bigg[\sum_{i=0}^{\kappa_n} f_2(S_n(i))\bigg] + \E\bigg[\bigg(\sum_{i=0}^{\kappa_n} f_1(S_n(i))\bigg)^2\bigg] - \E\bigg[\sum_{i=0}^{\kappa_n} f_1(S_n(i))^2\bigg].
\end{align*}
We finally compute these four expectations appealing to \eqref{eq:stick_breaking_discret}, \eqref{eq:moments_zeta} and Lemma \ref{lem:nombre_de_cycles_a_de_taille_k} below: as $n \to \infty$,
\begin{equation*}
\E\bigg[\sum_{i=0}^{\kappa_n} f_1(S_n(i))\bigg]
= n \E\bigg[\frac{f_1(\tilde{S}_n)}{\tilde{S}_n}\bigg]
= \sum_{\ell=1}^n \frac{f_1(\ell)}{\ell}
\sim \sum_{\ell=2}^n \frac{1}{\ln \ell}
\sim \frac{n}{\ln n};
\end{equation*}
similarly
\begin{equation*}
\E\bigg[\sum_{i=0}^{\kappa_n} f_2(S_n(i))\bigg]
= n \E\bigg[\frac{f_2(\tilde{S}_n)}{\tilde{S}_n}\bigg]
= \sum_{\ell=1}^n \frac{f_2(\ell)}{\ell}
\sim \sum_{\ell=2}^n \frac{\ell}{(\ln \ell)^2}
\sim \frac{n^2}{2 (\ln n)^2};
\end{equation*}
and
\begin{equation*}
\E\bigg[\bigg(\sum_{i=0}^{\kappa_n} f_1(S_n(i))\bigg)^2\bigg]
= n^2 \E\bigg[\bigg(\frac{f_1(\tilde{S}_n)}{\tilde{S}_n}\bigg)^2\bigg]
= n \sum_{\ell=1}^n \bigg(\frac{f_1(\ell)}{\ell}\bigg)^2
\sim n \sum_{\ell=2}^n \frac{1}{(\ln \ell)^2}
\sim \frac{n^2}{(\ln n)^2};
\end{equation*}
finally
\begin{equation*}
\E\bigg[\sum_{i=0}^{\kappa_n} f_1(S_n(i))^2\bigg]
= n \E\bigg[\frac{f_1(\tilde{S}_n)^2}{\tilde{S}_n}\bigg]
= \sum_{\ell=1}^n \frac{f_1(\ell)^2}{\ell}
\sim \sum_{\ell=2}^n \frac{\ell}{(\ln \ell)^2}
\sim \frac{n^2}{2 (\ln n)^2}.
\end{equation*}
Thus, the first two moments of $n^{-1} \ln n \sum_{i=0}^{\kappa_n} \zeta_i(S_n(i))$ converge to $1$, which implies \eqref{eq:nombre_de_coupes_pour_isoler_des_racines_dans_un_stick_breaking}, the convergence even holds in $L^2$.
\end{proofof}

In the course of the proof we used the following lemma.

\begin{lem}\label{lem:nombre_de_cycles_a_de_taille_k}
A size-biased pick $\tilde{S}_n$ from a discrete stick-breaking process $\mathbf{S}_n$ is uniformly distributed in $[n]$.
\end{lem}

\begin{proof}
As we have seen, $\mathbf{S}_n$ is distributed as the sizes of the subtrees of $\t_{n+1}$ after removing the root and its adjacent edges. Furthermore, as there are $n!$ recursive trees of size $n+1$, the latter are in bijection with permutations of $[n]$. Indeed, there is an explicit bijection between uniform random recursive trees of size $n+1$ and uniform random permutation of $[n]$, via the chinese restaurant process, see e.g. Goldschmidt and Martin \cite{Goldschmidt_Martin-Random_recursive_trees_and_the_Bolthausen-Sznitman_coalescent}, in which the vertex-sets of the subtrees of $\t_{n+1}$ after removing the root and its adjacent edges are exactly the cycles of the permutation. Then, for every $k \in [n]$,
\begin{equation*}
\P(\tilde{S}_n = k) = \sum_{i \ge 0} \frac{k}{n} \P(S_n(i) = k) = \frac{k}{n} \E[\Card\{i \ge 0 : S_n(i) = k\}],
\end{equation*}
and $\Card\{i \ge 0 : S_n(i) = k\}$ is distributed as the number of cycles of length $k$ in a uniform random permutation of $[n]$. It is well-known that the expectation of the latter is given by $1/k$ and the proof is complete.
\end{proof}

We end this section with three corollaries of Theorem \ref{thm:probabilite_site_uniforme_et_racine_dans_le_meme_arbre_ignifuge}. Denote by $f_{n,1}^\downarrow$ the size of the largest fireproof subtree of $\tn$.

\begin{cor}\label{cor:composantes_ignifugees_geantes_cas_sur_critique}
In the supercritical regime $p_n \ll \ln n / n$, we have
\begin{equation*}
n^{-1} f_{n,1}^\downarrow \cv 1
\quad\text{in probability.}
\end{equation*}
\end{cor}

\begin{proof}
Let $f_{n,1}^\downarrow \ge f_{n,2}^\downarrow \ge \dots \ge 0$ be the sizes of the fireproof subtrees of $\tn$, ranked in non-increasing order. Let also $X_n$ and $X_n'$ be two independent uniform random vertices in $[n]$, independent of the fire dynamics. Since $\sum_{i \ge 1} f_{n,i}^\downarrow \le n$, we have
\begin{align*}
\E[n^{-1} f_{n,1}^\downarrow]
&\ge \E\bigg[n^{-2} \sum_{i \ge 1} (f_{n,i}^\downarrow)^2\bigg]
\\
&= \P(X_n \text{ and } X_n' \text{ belong to the same fireproof component})
\\
&\ge \P(X_n, X_n' \text{ and } 1 \text{ belong to the same fireproof component})
\\
&\ge 2\P(X_n \text{ and } 1 \text{ belong to the same fireproof component})-1,
\end{align*}
and the latter converges to $1$ as $n \to \infty$ from Theorem \ref{thm:probabilite_site_uniforme_et_racine_dans_le_meme_arbre_ignifuge}. We conclude that $n^{-1} f_{n,1}^\downarrow$ converges to $1$ in probability.
\end{proof}

This further yields the following result for the subcritical regime.

\begin{cor}\label{cor:composantes_ignifugees_grandes_cas_sous_critique}
In the subcritical regime $1 \gg p_n \gg \ln n / n$, the sequence $(p_n f_{n,1}^\downarrow, n \ge 1)$ is tight.
\end{cor}

\begin{proof}
With the notations of the proof of Theorem \ref{thm:limite_nombre_sites_ignifuges_cas_sous_critique}, the root burns with high probability, so we implicitly condition on this event, and the number of edges fireproof in the root-component is given by $\sigma(n)$ which, rescaled by a factor $p_n$, converges in distribution towards $\e_1$. Then the argument of Lemma \ref{lem:proprietes_marche_aleatoire_limite_mesure} entails the joint convergence in distribution of the pair
\begin{equation*}
\bigg(p_n \sigma(n), \sum_{i=1}^{\sigma(n)} \delta_{p_n |\t_{n,i}'|}(\d x)\bigg)
\end{equation*}
towards $\e_1$ and a Poisson random measure with intensity $\e_1 x^{-2} \d x$ on $(0, \infty)$. In particular, for every $\varepsilon \in (0, 1)$, there exist two constants, say, $m$ and $M$, for which
\begin{equation*}
\P\bigg(m \le \max_{1 \le i \le \sigma(n)} p_n |\t_{n,i}'| \le M\bigg) > 1-\varepsilon,
\end{equation*}
for every $n$ large enough. Observe that
\begin{equation*}
p_n \frac{M/p_n}{\ln(M/p_n)} \cv 0,
\end{equation*}
and therefore a subtree which satisfies $m \le p_n |\t_{n,i}'| \le M$ is supercritical. It then follows from Corollary \ref{cor:composantes_ignifugees_geantes_cas_sur_critique} that such a subtree contains a fireproof component larger than $(1-\varepsilon)m/p_n$ (and smaller that $M/p_n$) with high probability and the proof is complete.
\end{proof}

Finally, in the critical regime, the behavior resembles that of sub or supercritical, according to the final state of the root.

\begin{cor}\label{cor:composantes_ignifugees_grandes_cas_critique}
Consider the critical regime $p_n \sim c \ln n / n$. We distinguish two cases:
\begin{enumerate}
\item On the event that the root burns, the sequence $((\ln n) n^{-1} f_{n,1}^\downarrow, n \ge 1)$ is tight.

\item On the event that the root is fireproof, $n^{-1} f_{n,1}^\downarrow$ converges to $1$ in probability.
\end{enumerate}
\end{cor}

\begin{proof}
For the first statement, on the event that the root burns, the number $\sigma(n)$ of edges fireproof in the root-component, rescaled by a factor $\ln n / n$ converges in distribution towards an exponential random variable with rate $c$ conditioned to be smaller than $1$. The rest of the proof follows verbatim from that of Corollary \ref{cor:composantes_ignifugees_grandes_cas_sous_critique} above.

For the second statement, we already proved in Proposition \ref{pro:composante_brulee_macro_contient_la_racine} that the probability that the root of $\tn$ is fireproof converges to $\ex^{-c}$ as $n \to \infty$. Thus, on this event, the probability that the root and an independent uniform random vertex belong to the same fireproof subtree converges to $1$ and the claim follows from the proof of Corollary \ref{cor:composantes_ignifugees_geantes_cas_sur_critique}.
\end{proof}

\section{On the sequence of burnt subtrees}\label{section5}

In the last section, we prove Theorem \ref{thm:limite_premiers_arbres_brules_cas_critique} which describes the asymptotic behavior of the burnt subtrees, in the critical regime $p_n \sim c \ln n / n$. One of the claims is that the probability that the root burns with the $j$-th fire converges to
\begin{equation*}
\E\bigg[\ex^{-\gamma_j} \prod_{i=1}^{j-1} (1 - \ex^{-\gamma_i})\bigg],
\end{equation*}
where $\gamma_0 = 0$ and $(\gamma_j - \gamma_{j-1})_{j \ge 1}$ is a sequence of i.i.d. exponential random variables with rate $c$. We can compute this expectation by writing $\E[\ex^{-\gamma_j} \prod_{i=1}^{j-1} (1 - \ex^{-\gamma_i})] = q_{j-1}-q_j$, where for every $j \ge 0$,
\begin{align*}
q_j
&= \E\bigg[\prod_{i=1}^j (1 - \ex^{-\gamma_i})\bigg]
\\
&= \E\bigg[\sum_{i=0}^j (-1)^i \sum_{1\le \ell_1 < \dots < \ell_i \le j} \ex^{-\gamma_{\ell_1}} \cdots \ex^{-\gamma_{\ell_i}}\bigg]
\\
&= \sum_{i=0}^j (-1)^i \sum_{1\le \ell_1 < \dots < \ell_i \le j} \E[\ex^{-i \gamma_{\ell_1}} \ex^{-(i-1) (\gamma_{\ell_2}-\gamma_{\ell_1})} \cdots \ex^{-(\gamma_{\ell_i} - \gamma_{\ell_{i-1}})}]
\\
&= \sum_{i=0}^j (-1)^i \sum_{1\le \ell_1 < \dots < \ell_i \le j} \bigg(\frac{c}{c+i}\bigg)^{\ell_1} \bigg(\frac{c}{c+i-1}\bigg)^{\ell_2-\ell_1} \cdots \bigg(\frac{c}{c+1}\bigg)^{\ell_i-\ell_{i-1}}.
\end{align*}

Before tackling the proof of Theorem \ref{thm:limite_premiers_arbres_brules_cas_critique}, let us give a consequence of this.

\begin{cor}
We have in the critical regime $p_n \sim c \ln n / n$:
\begin{equation*}
\lim_{k \to \infty} \lim_{n \to \infty} \P(\text{the root of } \tn \text{ is not burnt after } k \text{ fires})
= \lim_{n \to \infty} \P(\text{the root of } \tn \text{ is fireproof}),
\end{equation*}
and both are equal to $\ex^{-c}$.
\end{cor}

In particular, we see that for every $\varepsilon > 0$, there exists $k \in \N$ such that
\begin{equation*}
\liminf_{n \to \infty} \P(\text{the root of } \tn \text{ is fireproof} \mid \text{it is not burnt after } k \text{ fires}) > 1-\varepsilon.
\end{equation*}
In words, we can find $k$ large enough but independent of $n$, such that if the root of $\tn$ has not burnt after the first $k$ fires, then with high probability, it will be fireproof at the end of the dynamics.

\begin{proof}
Observe from Theorem \ref{thm:limite_premiers_arbres_brules_cas_critique} that for every $k \in \N$, the probability that the root of $\tn$ is not burnt after $k$ fires converges as $n \to \infty$ towards $q_k = \E[\prod_{i = 1}^k (1 - \ex^{-\gamma_i})]$, which in turn, thanks to the previous computation, converges as $k \to \infty$ towards
\begin{align*}
\E\bigg[\prod_{i = 1}^\infty (1 - \ex^{-\gamma_i})\bigg]
&= \sum_{i=0}^\infty (-1)^i \sum_{1\le \ell_1 < \dots < \ell_i} \bigg(\frac{c}{c+i}\bigg)^{\ell_1} \bigg(\frac{c}{c+i-1}\bigg)^{\ell_2-\ell_1} \cdots \bigg(\frac{c}{c+1}\bigg)^{\ell_i-\ell_{i-1}}
\\
&= \sum_{i=0}^\infty (-1)^i \frac{c}{i} \frac{c}{i-1} \cdots \frac{c}{1}
\\
&= \ex^{-c}.
\end{align*}
Finally, as we already mentioned, it follows from Proposition \ref{pro:composante_brulee_macro_contient_la_racine} that, as $n \to \infty$, the probability that the root of $\tn$ is fireproof converges to $\ex^{-c}$ as well.
\end{proof}

In order to prove Theorem \ref{thm:limite_premiers_arbres_brules_cas_critique}, we shall need the following three results on random recursive trees. Let $E_n$ be the set of all edges of $\tn$ which are not adjacent to the root. Fix $k \in \N$, consider a random recursive tree $\tn$ and a simple random sample $(e_{n,1}, \dots, e_{n,k})$ of $k$ edges from $E_n$. For each $i = 1, \dots, k$, denote by $v_{n,i}$ and $v_{n,i}'$ the two extremities of $e_{n,i}$ with the convention that $v_{n,i}$ is the closest one to the root. Finally, denote by $\rho_{n,k} \coloneqq \max_{1 \le i, j \le k} h(v_{n, i} \wedge v_{n, j})$, where $a \wedge b$ denotes the last common ancestor of $a$ and $b$ in $\tn$ and $h(a)$ the height of $a$ in $\tn$. Then in the complement of the ball centered at the root and of radius $\rho_{n,k}$, the paths from $1$ to $v_{n,i}$, $i = 1, \dots, k$ are disjoint.

\begin{lem}\label{lem:hauteurs_sites_internes_uniformes}
For every $i \in \N$ fixed, we have
\begin{equation*}
\frac{1}{\ln n} h(v_{n,i}) \cv 1
\quad\text{in probability.}
\end{equation*}
\end{lem}

\begin{proof}
We may replace $v_{n,1}$ by $v_{n,1}'$, which is uniformly distributed on the set of vertices of $\tn$ with height at least two. It is known that the degree of the root of $\tn$ rescaled by $\ln n$ converges to $1$ in probability so, as $n \to \infty$, $v_{n,1}'$ is close (in the sense of total variation) to a uniform random vertex on $[n]$. It follows that
\begin{equation*}
\frac{1}{\ln n} \ln v_{n,1} \cv 1
\quad\text{in probability.}
\end{equation*}
Finally, in a random recursive tree, the label $\ell$ of a vertex and its height $h(\ell)$ are related by
\begin{equation}\label{eq:hauteur_d_un_site_en_log}
\frac{1}{\ln \ell} h(\ell) \quad\mathop{\longrightarrow}^{}_{\ell \to \infty}\quad 1 \quad\text{in probability,}
\end{equation}
which concludes the proof.
\end{proof}

\begin{lem}\label{lem:k_sites_uniformes_se_separent_tot}
For every $k \in \N$ fixed, $\rho_{n,k} = O(1)$ in probability as $n \to \infty$.
\end{lem}

\begin{proof}
It suffices to consider $k = 2$; moreover, we may approximate $v_{n, 1}$ and $v_{n, 2}$ by a two independent uniform random vertices in $[n]$, say, $u_n$ and $v_n$. We have $h(u_n \wedge v_n) \ge 1$ if and only if $u_n$ and $v_n$ belong to the same tree-component in the forest, say, $(T_n^1, \dots, T_n^\kappa)$ obtained by removing the root and all its adjacent edges from $\tn$. We already noticed that this vector forms a discrete stick-breaking process on $[n-1]$. Let $\Upsilon_1 = 0$ and $\Upsilon_n \coloneqq \sum_{j=1}^\kappa |T_n^j|^2$ for $n \ge 1$; then, by decomposing according to the value of $|T_n^1|$, since the latter is uniformly distributed on $[n-1]$, we obtain
\begin{equation*}
\E[\Upsilon_n]
= \sum_{\ell=1}^{n-1} \frac{1}{n-1} (\ell^2 + \E[\Upsilon_{n-\ell}])
= (n-1) + \E[\Upsilon_{n-1}]
= \sum_{\ell=1}^n (\ell-1)
= \frac{n (n-1)}{2}.
\end{equation*}
Finally,
\begin{equation*}
\P(h(u_n \wedge v_n) \ge 1) = \frac{\E[\Upsilon_n]}{n^2} = \frac{n-1}{2n} \cv \frac{1}{2}.
\end{equation*}
We can iterate with the same reasoning: $h(u_n \wedge v_n) \ge \ell$ if and only if $u_n$ and $v_n$ belong to the same tree-component in the forest obtained by removing all the vertices at distance at most $\ell$ from the root and their adjacent edges. Then as previously,
\begin{equation*}
\P(h(u_n \wedge v_n) \ge \ell+1 \mid h(u_n \wedge v_n) \ge \ell) \cv \frac{1}{2}.
\end{equation*}
Therefore $h(u_n \wedge v_n)$ converges weakly to the geometric distribution with parameter $1/2$. It follows that for every $k \in \N$ fixed, the sequence $(\rho_{n,k})_{n \in \N}$ is bounded in probability.
\end{proof}

\begin{lem}\label{lem:taille_sous_arbre_engendre_par_k}
For each $v \in [n]$, denote by $\tn(v)$ the subtree of $\tn$ that stems from $v$. Then for any sequence of integers $1 \ll v_n \ll n$, we have
\begin{equation*}
\frac{v_n}{n} |\tn(v_n)| \cvloi \e_1,
\end{equation*}
where $\e_1$ is an exponential random variable with rate $1$.
\end{lem}

\begin{proof}
We interpret the size of $\tn(v)$ in terms of a P\'{o}lya urn. Recall the iterative construction of random recursive trees described in the introduction. At the $v$-th step, we add the vertex $v$ to the current tree; then cut the edge which connects $v$ to its parent to obtain a forest with two components with sets of vertices $\{v\}$ and $[v-1]$ respectively. Next, the vertices $v+1, \dots, n$ are added to this forest independently one after the others, and the parent of each is uniformly chosen in the system. Considering the two connected components, we see that their sizes evolve indeed as an urn with initial configuration of $1$ red ball and $v-1$ black balls and for which at each step, a ball is picked uniformly at random and then put back in the urn, along with one new ball of the same color. Then $|\tn(v)| - 1$ is equal to the number of ``red'' outcomes after $n-v$ trials and this is known to have the beta-binomial distribution with parameters $(n-v, 1, v-1)$, i.e.
\begin{equation*}
\P(|\tn(v)| = \ell+1)
= (v-1) \frac{(n-v)!}{(n-v-\ell)!} \frac{(n-\ell-2)!}{(n-1)!},
\qquad \ell = 0, \dots, n-v.
\end{equation*}
Using Stirling formula, we compute for every $x \in (0, \infty)$ and $1 \ll v_n \ll n$,
\begin{equation*}
\P(|\tn(v_n)| = \lfloor x n / v_n \rfloor) = \frac{v_n}{n} \ex^{-x} (1+o(1)),
\end{equation*}
and the claim follows from this local convergence.
%
\end{proof}

\begin{figure}[!h] 
\begin{center}
\includegraphics[width=\linewidth]{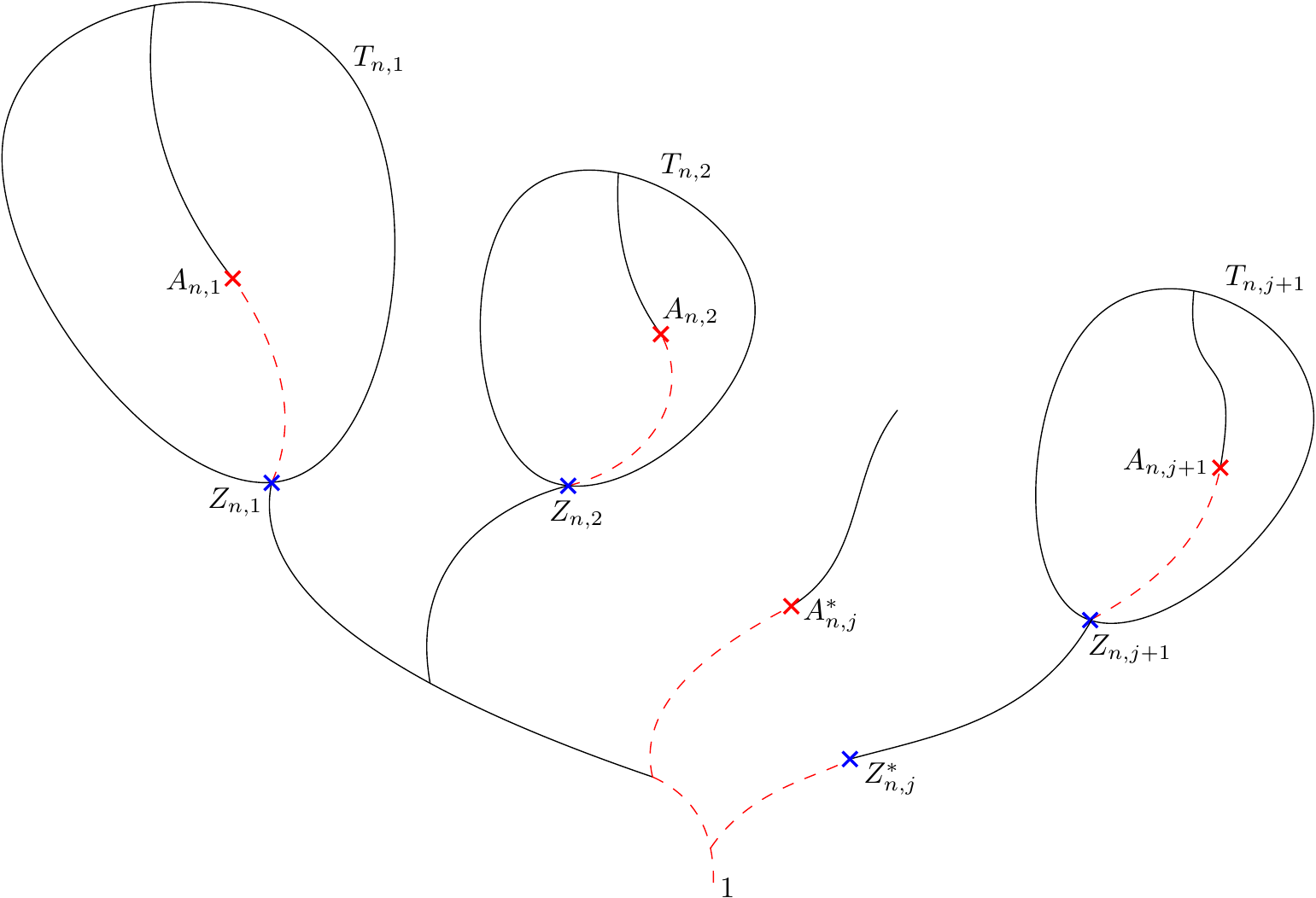}
\caption{Illustration of the proof of Theorem \ref{thm:limite_premiers_arbres_brules_cas_critique}.}
\end{center}
\end{figure}

The proof of Theorem \ref{thm:limite_premiers_arbres_brules_cas_critique} consists of four main steps, which are indicated by ``\textsc{Step 1}'', \ldots, ``\textsc{Step 4}''. We first consider the case of the root: we view the fire dynamics as a dynamical percolation in continuous-time where each fireproof edge is deleted and each burnt component is discarded. Then the results of Bertoin \cite{Bertoin-Sizes_of_the_largest_clusters_for_supercritical_percolation_on_random_recursive_trees} allow us to derive the size of the root-cluster at the instant of the first fire, which gives us the probability that this cluster burns with the first fire. If it does not, then we discard a small cluster (this is proven in step 2) and then continue the percolation on the remaining part until the second fire; again we know the size of the root-cluster at this instant. By induction, we obtain for every $j \ge 1$ the probability that the root burns with the $j$-th fire, and the size of its burnt component.

In a second step, we investigate the size of the first burnt subtree, conditionally given that it does not contain the root. We denote by $A_{n,1}$ the closest extremity to the root of the first edge which is set on fire. Since the root does not burn at this instant, there exists at least one fireproof edge on the path from the root to $A_{n, 1}$; consider all the vertices on this path which are adjacent to such a fireproof edge and let $Z_{n,1}$ be the closest one to $A_{n,1}$. Finally, let $\t_{n, 1}$ be the subtree of $\tn$ that stems from $Z_{n,1}$. Then $b_{n,1}$ is the size of the first burnt subtree of $\t_{n,1}$ and the latter contains $A_{n,1}$ and its root $Z_{n,1}$. We estimate the size of $\t_{n,1}$; further, we know the number of fireproof edges in $\t_{n,1}$ before the first fire and, thanks to the first step (recall that, conditionally given its size, $\t_{n,1}$ is a random recursive tree), we know the size of the root-component which burns with the first fire.

In the third step, we extend the results of the second one to the first two burnt components, conditionally given that none of them contains the root. The reasoning is the same as for the second step; we prove that the paths between the root of $\tn$ and the starting points of the first two fires, respectively $A_{n,1}$ and $A_{n,2}$, become disjoint close to the root so that the dynamics on each are essentially independent: in particular the variables $Z_{n,1}$ and $Z_{n,2}$ (the latter plays the same role as $Z_{n,1}$ for the second fire) become independent at the limit. We conclude by induction that the estimate holds for the sizes of the first $k$ burnt subtrees, conditionally given that the root does not burn with any of the first $k$ fires.

The last step is a simple remark: the fact that the root burns does not affect the previous reasoning. Indeed, if the root burns with, say, the $j$-th fire, then there exists a vertex $Z_{n,j}^\ast$ between the root and the starting point $A_{n,j+1}$ of the $j+1$-first fire where the $j$-th fire stops. Therefore the estimate for the size of the burnt components before the instant where the root burns holds also for the burnt components which come after that the root has burnt.

\begin{proofof}{Theorem \ref{thm:limite_premiers_arbres_brules_cas_critique}}
\textsc{Step 1:} Let us first consider the root-component. It will be more convenient to work in a continuous-time setting. We attach to each edge $e$ of $\tn$ two independent exponential random variables, say, $\e^{(f)}(e)$ and $\e^{(b)}(e)$, with parameter $(1-p_n) / \ln n$ and $p_n / \ln n$ respectively. They should be thought of as the time at which the edge $e$ becomes fireproof or is set on fire respectively; if the edge is already burnt because of the propagation of a prior fire, we do not do anything; likewise, if an edge $e$ is fireproof, we do not set it on fire at the time $\e^{(b)}(e) > \e^{(f)}(e)$. Then the time $\tau(n)$ corresponding to the first fire is given by
\begin{equation*}
\tau(n) = \inf\{\e^{(b)}(e) : \e^{(b)}(e) < \e^{(f)}(e)\}.
\end{equation*}
By the properties of exponential distribution, the variable $\inf_e \e^{(b)}(e)$ is exponentially distributed with parameter $(n-1)p_n / \ln n \to c$ as $n \to \infty$. Denote by $\hat{e}$ the edge of $\tn$ such that $\e^{(b)}(\hat{e}) = \inf_e \e^{(b)}(e)$. Then $\e^{(f)}(\hat{e})$ is exponentially distributed with parameter $(1-p_n) / \ln n \ll 1$ and so $\e^{(b)}(\hat{e}) < \e^{(f)}(\hat{e})$ with high probability. We conclude that $\tau(n)=\inf_e \e^{(b)}(e)$ with high probability and the latter converges in distribution to $\gamma_1$. It then follows from Corollary 2 of Bertoin \cite{Bertoin-Sizes_of_the_largest_clusters_for_supercritical_percolation_on_random_recursive_trees} that the size of the root-component at the instant $\tau(n)$, rescaled by a factor $n^{-1}$, converges to $\ex^{-\gamma_1}$ as $n \to \infty$. We conclude that the probability that the root of $\tn$ burns with the first fire converges to $\E[\ex^{-\gamma_1}]$ and, conditional on this event, the size of the corresponding burnt component rescaled by a factor $n^{-1}$ converges to $\ex^{-\gamma_1}$ in distribution.

If the root does not burn with the first fire, then we shall prove in the next step that the size $b_{n,1}$ of the first burnt component is negligible compared to $n$ with high probability. The previous reasoning then shows that the time of the second fire converges in distribution to $\gamma_2$, the probability that the root of $\tn$ burns at the second fire converges to $\E[(1 - \ex^{-\gamma_1}) \ex^{-\gamma_2}]$ and, on this event, the size of the corresponding burnt component rescaled by a factor $n^{-1}$ converges to $\ex^{-\gamma_2}$ in distribution. Again, if the root does not burn with the second fire, then the second burnt component is negligible compared to $n$ with high probability and the general claim follows by induction.

\textsc{Step 2:} For the rest of this proof, we condition on the event that the root of $\tn$ burns with the $j$-th fire with $j \ge 1$ fixed. We first prove the convergence of the logarithm of the sizes of the first $j-1$ burnt subtrees. Observe that the $j-1$ first edges which are set on fire are distributed as a simple random sample of $j-1$ edges from the set $E_n$ of edges of $\tn$ not adjacent to the root. Lemma \ref{lem:hauteurs_sites_internes_uniformes} then entails that
\begin{equation*}
\frac{1}{\ln n} (h(A_{n,1}), \dots, h(A_{n,j-1})) \cv (1, \dots, 1)
\quad\text{in probability.}
\end{equation*}
Consider first the first burnt subtree. The number $\theta_{n,1}$ of fireproof edges when the first edge is set on fire follows the geometric distribution with parameter $p_n \sim c \ln n / n$, truncated at $n-1$, so
\begin{equation}\label{eq:convergence_temps_premier_feu}
\frac{\ln n}{n} \theta_{n,1} \cvloi \gamma_1.
\end{equation}
Conditioning the root not to burn with the first fire amounts to conditioning the path from $1$ to $A_{n,1}$ to contain at least one of the $\theta_{n,1}$ first fireproof edges. Since, conditionally given $\theta_{n,1}$, these edges are distributed as a simple random sample from the $n-2$ edges of $\tn$ different from the first one which is set on fire, then for every $x \in (0, 1)$, the probability that $d(A_{n,1}, Z_{n,1})$ is smaller than $x \ln n$, conditionally given that it is smaller than $h(A_{n,1})$ is given by
\begin{equation*}
\E\Bigg[\frac{1 - (1 - \frac{\lfloor x \ln n \rfloor}{n-2}) \cdots (1 - \frac{\lfloor x \ln n \rfloor}{n-\theta_{n,1}-1})}
{1 - (1 - \frac{h(A_{n,1})}{n-2}) \cdots (1 - \frac{h(A_{n,1})}{n-\theta_{n,1}-1})}\Bigg]
\sim \E\Bigg[\frac{1 - (1 - \frac{x \ln n}{n})^{\theta_{n,1}}}{1 - (1 - \frac{\ln n}{n})^{\theta_{n,1}}}\Bigg]
\mathop{\longrightarrow}^{}_{n \to \infty} \E\bigg[\frac{1 - \exp(- x \gamma_1)}{1 - \exp(- \gamma_1)}\bigg].
\end{equation*}
We conclude that
\begin{equation*}
\bigg(\frac{\ln n}{n} \theta_{n,1}, \frac{1}{\ln n} d(A_{n,1}, Z_{n,1})\bigg) \cvloi (\gamma_1, Z_1).
\end{equation*}
Appealing to \eqref{eq:hauteur_d_un_site_en_log}, we further have
\begin{equation*}
\frac{1}{\ln n} \ln Z_{n,1} \cvloi 1 - Z_1,
\end{equation*}
jointly with the convergence \eqref{eq:convergence_temps_premier_feu}. Note that with the notation of Lemma \ref{lem:taille_sous_arbre_engendre_par_k}, we have $\t_{n,1} = \tn(Z_{n,1})$. Since $1 \ll Z_{n,1} \ll n$ in probability, we obtain
\begin{equation*}
\frac{1}{\ln n} \ln |\t_{n,1}| \cvloi Z_1,
\end{equation*}
again jointly with \eqref{eq:convergence_temps_premier_feu}. Consider finally the fire dynamics on $\t_{n,1}$ and denote by $N_{n,1}$ the number of edges which are fireproof in this tree before the first fire. Note that conditionally given $N_{n,1}$, these edges are distributed as a simple random sample of $N_{n,1}$ edges from the complement in $\t_{n,1}$ of the path from its root $Z_{n,1}$ to $A_{n,1}$ and that the fire burns this path. Conditionally given $\theta_{n,1}$, $d(A_{n,1}, Z_{n,1})$ and $|\t_{n,1}|$, the variable $N_{n,1}$ has a hypergeometric distribution: it is the number of edges picked amongst the $|\t_{n,1}| - 1 - d(A_{n,1}, Z_{n,1}) \sim |\t_{n,1}|$ ``admissible'' edges from the $n-1$ edges of $\tn$ after $\theta_{n,1}$ draws without replacement. Since $\theta_{n,1} = o(n)$ in probability, conditionally given $\theta_{n,1}$, $d(A_{n,1}, Z_{n,1})$ and $|\t_{n,1}|$, the variable $N_{n,1}$ is close (in total variation) to a binomial variable with parameter $\theta_{n,1}$ and $n^{-1} |\t_{n,1}|$. It is easy to check that a binomial random variable with parameters, say, $n$ and $p(n)$, rescaled by a factor $(n p(n))^{-1}$, converges in probability to $1$; it follows from the previous convergences that \begin{equation*}
\frac{\ln |\t_{n,1}|}{|\t_{n,1}|} N_{n,1}
= \frac{\ln n}{n} \theta_{n,1} \frac{\ln |\t_{n,1}|}{\ln n} \frac{n}{\theta_{n,1} |\t_{n,1}|} N_{n,1}
\cvloi \gamma_1 Z_1.
\end{equation*}
We see that in $\t_{n,1}$, we fireproof $N_{n,1} \approx \gamma_1 Z_1 |\t_{n,1}| / \ln |\t_{n,1}|$ edges before setting the root on fire. From the first step of the proof, the size $b_{n,1}$ of the first burnt component (which, by construction, contains the root of $\t_{n,1}$) is comparable to $|\t_{n,1}|$; in particular,
\begin{equation*}
\frac{1}{\ln n} \ln b_{n, 1} \cvloi Z_1.
\end{equation*}

\textsc{Step 3:} Consider next the first two fires and, again, condition the root not to be burnt after the second fire; in particular, there exists at least one fireproof edge on the path from $1$ to $A_{n,1}$ and on that from $1$ to $A_{n,2}$. Thanks to Lemma \ref{lem:k_sites_uniformes_se_separent_tot}, we have $h(A_{n,1} \wedge A_{n,2}) = o(\ln n)$ in probability. Then with high probability, $Z_{n,1}$ and $Z_{n,2}$ are located outside the ball centered at $1$ and of radius $h(A_{n,1} \wedge A_{n,2})$ where the two paths from $1$ to $A_{n,1}$ and to $A_{n,2}$ are disjoint so where the location of the fireproof edges are independent (conditionally given $\theta_{n,1}$ and $\theta_{n,2}$). With the same reasoning as for the first fire, we obtain
\begin{equation*}
\bigg(\frac{\ln n}{n} (\theta_{n,1}, \theta_{n,2}), \frac{1}{\ln n} \big( d(A_{n,1}, Z_{n,1}), d(A_{n,2}, Z_{n,2}) \big)\bigg)
\cvloi
\big((\gamma_1, \gamma_2), (Z_1, Z_2)\big).
\end{equation*}
Then
\begin{equation*}
\frac{1}{\ln n} (\ln Z_{n,1}, \ln Z_{n,2}) \cvloi (1 - Z_1, 1 - Z_2),
\end{equation*}
and
\begin{equation*}
\frac{1}{\ln n} (\ln |\t_{n,1}|, \ln |\t_{n,1}|) \cvloi (Z_1, Z_2).
\end{equation*}
Finally,
\begin{equation*}
\frac{1}{\ln n} (\ln b_{n, 1}, \ln b_{n, 2}) \cvloi (Z_1, Z_2).
\end{equation*}
We conclude by induction that
\begin{equation*}
\frac{1}{\ln n} (\ln b_{n, 1}, \dots, \ln b_{n, j-1}) \cvloi (Z_1, \dots, Z_{j-1}).
\end{equation*}

\textsc{Step 4:} Consider next the $j+1$-st fire. The number of flammable edges after the $j$-th fire is given by $q_{n,j} \coloneqq n - (1 + \theta_{n,j} + (b_{n,1} - 1) + \dots + (b_{n,j} - 1))$ which, rescaled by a factor $n^{-1}$, converges in distribution towards $1 - \ex^{-\gamma_j}$. As previously, $\theta_{n,j+1} - \theta_{n,j}$ is then distributed as a geometric random variable with parameter $p_n \sim c \ln n / n$ and truncated at $q_{n,j} \gg n / \ln n$. It follows that
\begin{equation*}
\frac{\ln n}{n} (\theta_{n,1}, \dots, \theta_{n,j+1}) \cvloi (\gamma_1, \dots, \gamma_{j+1}).
\end{equation*}
Moreover since the root has burnt, there is at least one fireproof edge on the path from $1$ to $A_{n, j+1}$, the starting point of the $j+1$-st fire. All the previous work then applies and we obtain the convergence
\begin{equation*}
\frac{1}{\ln n} (\ln b_{n, 1}, \dots, \ln b_{n, j-1}, \ln b_{n, j+1}) \cvloi (Z_1, \dots, Z_{j-1}, Z_{j+1}).
\end{equation*}
The general claim follows by a last induction.
\end{proofof}


\vspace{\baselineskip}
\textbf{Acknowledgement. }
I thank Jean Bertoin for reading the manuscript and for stimulating discussions, as well as two anonymous referees for their comments which helped improving this paper. This work is supported by the Swiss National Science Foundation 200021\_144325/1.


{\small
\bibliographystyle{acm}

}

\end{document}